\documentclass[11pt,english,twoside]{amsart}
\usepackage[a4paper, left=3.5cm, right=2cm]{geometry}
\usepackage{babel}
\usepackage[latin1]{inputenc}
\usepackage{graphicx}
\usepackage{epsfig}
\usepackage{color}
\usepackage{amsmath}
\usepackage{bbm}
\usepackage{sidecap}
\usepackage{comment}
\usepackage{bm}
\usepackage{enumerate}
\usepackage{amsfonts}
\usepackage{amssymb}
\catcode`\@=11
\catcode`\@=12
\usepackage{colordvi}
\usepackage{psfrag, wrapfig}
\usepackage{subfigure}
\usepackage{stmaryrd}
\usepackage{bm}
\usepackage{multirow}
\usepackage{graphicx}
\usepackage{color}
\usepackage{fancyhdr}
\def\bea{\begin{eqnarray}}
\def\eea{\end{eqnarray}}
\newtheorem{theorem}{Theorem}[section]
\newtheorem{lemma}[theorem]{Lemma}
\newtheorem{remark}[theorem]{Remark}
\newtheorem{cor}[theorem]{Corollary}
\newtheorem{definition}[theorem]{Definition}
\newtheorem{example}[theorem]{Example}
\newtheorem{algorithm}[theorem]{Algorithm}

\newtheorem{assumptions}[theorem]{Assumptions}

\newcommand{\f}{f}
\newcommand{\T}{\mathcal T}
\renewcommand{\P}{\mathcal P}
\newcommand{\R}{\mathbb{R}}

\newcommand{\aquestion}[1] { }

\newcommand{\noshow}[1] {  }
\newcommand{\matlab}[1] {  }

\newcommand{\dA}{}
\newcommand{\dS}{}

\def\ud{\underline D}
\def\ol{\overline}
\def\H{\mathcal{H}}
\def\veps{\varepsilon}
\def\R{\mathbb{R}}
\begin{document}
\title{Scalar conservation laws on moving hypersurfaces}
\author{Gerhard Dziuk, Dietmar Kr\"oner, Thomas M\"uller}%\\[3mm]
\address{Abteilung f\"ur Angewandte Mathematik,
University of Freiburg,
Hermann-Herder-Stra{\ss}e 10,
D--79104 Freiburg i. Br.,
Germany}
\thanks{The work has been supported by
Deutsche Forschungsgemeinschaft via SFB/TR 71 `Geometric Partial
Differential Equations'}
\date{}
%\subjclass{65}
%\keywords{}
\maketitle

\newcommand{\pd}[2]{\frac{\partial #1}{\partial #2}}
\parindent0.0em

\begin{abstract}
We consider conservation laws on moving hypersurfaces.
In this work the velocity of the surface is prescribed.
But one may think of the velocity to be given by PDEs in the bulk phase.
%while in forthcoming papers we will consider the case for which the velocity is
%also an unknown function of the problem and may depend on the flows in the neighboring bulk phases.
We prove existence and uniqueness for a scalar conservation law on the moving
surface. This is done via a parabolic regularization of the hyperbolic PDE. We then prove
suitable estimates for the solution of the regularized PDE, that are independent
of the regularization parameter. We introduce the concept of an entropy solution for
a scalar conservation law on a moving hypersurface.
We also present some numerical experiments. As in the Euclidean case we expect discontinuous solutions,
in particular shocks. It turns out that in addition to the "Euclidean shocks" geometrically induced shocks may appear.
\end{abstract}

\section{Introduction}
The theoretical and numerical solution of partial differential equations on stationary
or moving surfaces has become quite important during the last decade. In many applications
PDEs in bulk phases are coupled to PDEs on interfaces between these phases. There is
a satisfactory analysis and numerical analysis for elliptic and parabolic equations on
stationary or moving surfaces. For references we refer to \cite{dziuk1}, \cite{dziuk2}, \cite{dziuk4}.
Several phenomena like shallow water equations on the earth,
relativistic flows, transport processes on surfaces, transport of oil
on the waves of the ocean or the transport on moving interfaces between
two fluid  are modeled by transport equations, and thus hyperbolic PDEs, on fixed or moving surfaces.
These equations often are highly nonlinear.

In this work we study scalar conservation laws on moving hypersurfaces without boundary in $\R^{n+1}$.
The motion of the surface is prescribed.
Assume that $\Gamma(t)$ is a family of smooth and compact hypersurfaces which moves
smoothly with time $t\in[0,T]$.
When the scalar material quantity $u=u(x,t)$, $x\in\Gamma(t)$, $t\in [0,T]$, is propagated with
the surface and simultaneously transported via a given flux $f=f((x,t),u)$ on the surface,
then its evolution with respect to prescribed initial values $u_0$ is governed by the initial value problem
\begin{equation}\label{pde}
\dot u +u \nabla_{\Gamma}\cdot v +\nabla_{\Gamma} \cdot f(\cdot,u) = 0 \,\,\,\,\,\,\, \mbox{on} \,\, G_T,\qquad
u(\cdot,0)=u_{0} \,\, \mbox{on} \,\,\Gamma_0.
\end{equation}
Here $v$ denotes the velocity of the surface $\Gamma$, and $\nabla_\Gamma$ is the surface
gradient. The dot stands for a material derivative.
$f$ is the given flux function which we assume to be tangentially divergence free on $\Gamma$
and which is a tangent vector to the surface. By $G_T$ we denote the space time surface
\begin{equation}\label{space-time}
G_T=\bigcup_{t\in (0,T)} \Gamma(t)\times \{t\}.
\end{equation}
The quantities appearing in the PDE (\ref{pde}) are well defined for $u:G_T\to\R$ and do not
depend on the ambient space.

%We will make use of the embeddedness of $G_T$ in order to have transparent arguments.

\begin{comment}
The derivatives $\dot g$  and $\nabla_{\Gamma} g $ are defined in (\ref{TangentialGradient}) and (\ref{MatDer}) respectively.
We assume that the surface  $\Gamma_t \subset \R^{n+1}$ (i.e. $\Phi$) is sufficiently smooth with respect to space and time and that it does not change its topological properties during the evolution.   Here the velocity $v$ describes the evolution of the surface and is given while in forthcoming papers we will consider the case for which the velocity $v$ is also an unknown function of the problem and may depend on the flows of the neighboring bulk phases.
The aim of this paper is to prove the existence and uniqueness of an entropy solution of (\ref{pde}), (\ref{initialvalue}), to develop a numerical scheme (finite volume scheme) for computing approximate solutions and to show some numerical experiments.
\end{comment}
%Notice that u depends on $\varepsilon$: $u=u_\varepsilon$.

Let us briefly summarize the published results related to this topic.  Total variation estimates for time independent Riemannian manifolds can be found in \cite{kk2}.
The existence proof of entropy solutions on time independent Riemannian manifolds is considered in \cite{ben-lef} by viscous approximation. The ideas are based on Kruzkov's and DiPerna's theories for the Euclidean case.
In a forthcoming paper Lengeler and one of the authors \cite{len-mue} are generalizing the results which we are going to prove in this contribution to the case of time dependent Riemannian manifolds. They show existence and uniqueness (in the space of measure-valued entropy
solutions) of entropy solutions for initial values in $L^\infty$ and derive total variation estimates if the initial values are in $\operatorname{BV}$.
Convergence of finite volume schemes on time independent Riemannian manifolds can be found in \cite{amo-ben-lef}.
In the paper \cite{lef-oku-nev} LeFloch, Okutmustur and Neves  prove an error estimate of the form $||u-u_h||_{L^1}\le c h^{\frac{1}{4}}$ for the scheme in \cite{amo-ben-lef}. The proof generalizes the ideas of the Euclidian case and the convergence rate is the same. This result was generalized to the time dependent case by Giesselmann \cite{gie} under the assumption that an entropy solution exists, which we are going to prove in this paper. In \cite{amo-lef-nev} an error estimate for hyperbolic conservation laws on an $(N+1)$ dimensional manifold (spacetime), whose flux is a field of differential forms of degree $N$, is shown. The matter Einstein equation for  perfect fluids on  spacetimes  in the context of general relativity is considered in \cite{lefl-stew}. A wave propagation algorithm for hyperbolic systems on curved manifolds with application in relativistic hydrodynamics and  magnetohydrodynamics have been developed and tested in \cite{lev1}, \cite{lev2}, \cite{lev5}, and finite volume scheme on 
spherical 
domains, partially with adaptive grid refinement in \cite{lev4}, \cite{lev5}.

This paper is organized as follows. In Section 2 we will summarize the notations and basic relations for moving hypersurfaces which we need to show existence for \eqref{pde}. The PDE in \eqref{pde} will be derived in Section 3. Since the weak solution of \eqref{pde} is in general not unique we will define entropy solutions in Section 4. The  idea for the existence proof is based on the approximation of the solution of \eqref{pde} by the solution of a parabolic regularization, which will be presented in Section 5. In Section 6 and 7 we prove uniform estimates in the $H^{1,1}$-norm of the solution of the parabolic regularization. This implies compactness in $L^1$ and therefore existence for \eqref{pde}, which is the main subject of Section 8. Since this existence result  depends on the special regularization, defined in Section 5, we have to prove uniqueness of \eqref{pde} in Section 9. With the numerical algorithm described in Section 10 we have performed some numerical experiments. The results are shown in 
Section 11.

\vspace{0.5cm}

%%%%%%%%%%%%%%%%%%%%%%%%%%%%%%%%%%%%%aaaaaaaaaaaaaaaaaaaaaaaaaaaaaaaaaaaaaaaaaa
%%%%%%%%%%%%%%%%%%%%%%%%%%%%%%%%%%%%%eeeeeeeeeeeeeeeeeeeeeeeeeeeeeeeeeeeeeeeeeee
\section{Notations and basic relations for moving hypersurfaces}
In this Section we present the description of the moving geometry. We use the
notion of tangential or surface gradients.
%\subsection{Basic assumptions}
% ~\\
% 
% {\sf ????????    Den Paragrafen 2.1 wuerde ich gern am Ende fertig stellen, wenn klar ist, was
% wir wirklich brauchen. Ich denke, dass wir auch etwas zur Regularit\"at der
% Fl\"achen sagen m\"ussen.(Gerd)}

\begin{assumptions} \label{ass1} Let $\Gamma_t=\Gamma(t)\subset \R^{n+1}$ for $t\in [0,T]$ be a  time dependent, closed, smooth hypersurface.
The initial surface $\Gamma_0$ is transported by the smooth function
\begin{align}
 \Phi: \Gamma_0 \times [0,T] \rightarrow \R^{n+1}
 \end{align}
 with $\Phi(\Gamma_0,t)=\Gamma_t$ and $\Phi(\cdot,0)=Id$.
 We assume that $\Phi(\cdot,t):\Gamma_0\rightarrow \Gamma(t)$ is a diffeomorphism for every $t\in[0,T]$.
 The velocity of the material points is denoted by $v:=\partial_t\Phi\circ\Phi^{-1}$.
 The tangential flow of a conservative material quantity $u$ with $u(\cdot,t):\Gamma_t\rightarrow \R$ is described by
 a flux function $f=f((x,t), u)$ which is a family of vector fields such that $f((x,t), u)$ is a tangent vector
 at the surface $\Gamma_t$ for $x\in \Gamma_t$, $t\in[0,T]$ and $ u \in \R$. We assume that all derivatives of f are bounded and that  $\nabla_{\Gamma} \cdot f((\cdot,t),s)=0$ for all fixed $t\in \mathbb{R}^+,s\in \mathbb{R}$. The definition of $\nabla_{\Gamma}\cdot $ is given below.
 \end{assumptions}

\subsection{Tangential derivatives and geometry}
Let us assume that $\Gamma$ is a compact $C^2$-hypersurface in
$\R^{n+1}$ with normal vector field $\nu$.
\begin{definition}\label{TanDef}
For a differentiable function $g:\Gamma\to\R$ we define its
tangential gradient as
\begin{equation}\label{TangentialGradient}
\nabla_\Gamma g=\nabla \overline g-\nabla\overline g\cdot \nu \,\nu,
\end{equation}
where $\overline g$ is an extension of $g$ to a neighbourhood of $\Gamma$.
We denote the components of the gradient by
$$\nabla_\Gamma g=
\left(\ud_1 g,\dots,\ud_{n+1}g \right).$$
\end{definition}
The Laplace-Beltrami operator then is given by
$$\Delta_\Gamma g = \nabla_\Gamma\cdot\nabla_\Gamma g=\sum_{j=1}^{n+1}\underline D_j\underline D_j g.$$
It is well known that the tangential gradient only depends on the values of $g$ on $\Gamma$.
For more informations about this concept we refer to \cite{Deckelnick-Dziuk-Elliott-Acta}.
With the help of tangential gradients we can describe the geometric
properties of $\Gamma$.
The matrix
\begin{equation*}
\mathcal{H}=\nabla_\Gamma\nu,\quad
\mathcal{H}_{ij}=\left(\nabla_\Gamma\nu\right)_{ij}=\ud_i\nu_j
=\ud_j\nu_i\quad (i,j=1,\dots,n+1)
\end{equation*}
has a zero eigenvalue in normal direction: $\mathcal{H}\nu=0$.
The remaining eigenvalues $\kappa_1,\dots,\kappa_n$ are the
principal curvatures of $\Gamma$. We can view the matrix
$\mathcal{H}$ as an extended Weingarten map. The mean curvature $H$
of $\Gamma$ is given as the trace of $\mathcal{H}$,
\begin{equation*}
H=\sum_{j=1}^{n+1}\mathcal{H}_{jj}=\sum_{j=1}^n\kappa_j,
\end{equation*}
where we note, that this definition of the mean curvature
differs from the common definition by a factor $\frac 1n$.
Integration by parts on a hypersurface $\Gamma$ is given by
the following formula. A proof for surfaces without boundary
can be found in \cite{GT}. The extension to surfaces with
boundary is easily obtained. By $\mu$ we denote the conormal
to $\partial \Gamma$.
\begin{lemma}\label{IntPartsLemma}
\begin{equation}\label{IntParts}
\int_\Gamma \nabla_\Gamma g \,\dA=\int_\Gamma gH\nu\,\dA+
\int_{\partial \Gamma}g\mu \,\dS.
\end{equation}
\end{lemma}
Higher order tangential derivatives do not commute. But we have
the following law for second derivatives. Here and in the following
we use the summation convention that we sum over doubly appearing
indices.
\begin{lemma}\label{VertauschungLemma}
For a function $g\in C^2(\Gamma)$ we have for $i,k=1,\dots,n+1$, that
\begin{equation}\label{vertauschung}
\ud_i\ud_kg=\ud_k\ud_ig+\H_{kl}\ud_lg\nu_i-\H_{il}\ud_lg\nu_k.
\end{equation}
\end{lemma}
For the convenience of the reader we give a short proof for this
relation.
\begin{proof} We use the definition (\ref{TangentialGradient})
of the tangential gradient and extend $g$ constantly in
normal direction to obtain $\overline{g}$. Then
\begin{eqnarray*}
&&\ud_i\ud_kg-\ud_k\ud_ig=
\ud_i\left(\ol g_{x_k}-\ol g_{x_l}\nu_l\nu_k\right)
-\ud_k\left(\ol g_{x_i}-\ol g_{x_m}\nu_m\nu_i\right)\\
&&=\ol g_{x_kx_i}-\ol g_{x_kx_r}\nu_r\nu_i
-\left(\ol g_{x_lx_i}-\ol g_{x_lx_s}\nu_s\nu_i\right)\nu_l\nu_k
-\ol g_{x_l}\left(\mathcal{H}_{il}\nu_k+\nu_l\mathcal{H}_{ik}\right)\\
&&-\ol g_{x_ix_k}+\ol g_{x_ix_r}\nu_r\nu_k
+\left(\ol g_{x_mx_k}-\ol g_{x_mx_s}\nu_s\nu_k\right)\nu_m\nu_i
+\ol g_{x_m}\left(\mathcal{H}_{mk}\nu_i+\nu_m\mathcal{H}_{ik}\right)\\
&&=\ol g_{x_m}\mathcal{H}_{mk}\nu_i-\ol g_{x_l}\mathcal{H}_{il}\nu_k
=\mathcal{H}_{kl}\nu_i\ud_{l} g-\mathcal{H}_{il}\nu_k\ud_l g.
\end{eqnarray*}
In the last step we have used that $\mathcal{H}\nu=0$.
\end{proof}
\subsection{Material derivatives}
In this Section we work with moving surfaces.
\begin{definition}\label{MatDerDef}
For a differentiable function $g:G_T\to\R$ we define the material derivative
\begin{equation}\label{MatDer}
\dot g=\frac{\partial g}{\partial t}+v\cdot\nabla g.
\end{equation}
\end{definition}
Note that the material derivative only depends
on the values of $g$ on the space-time surface
%\begin{equation}\label{SpaceTimeSurface}
$G_T$.
%\end{equation}
In our proofs we will frequently use the following formula for
the commutation of spatial (tangential) derivatives and (material)
time derivatives.
A proof is given in the Appendix.
\begin{lemma}\label{dotgradLemma} For $g\in C^2(G_T)$ we have that
\begin{equation}\label{dotgrad}
(\ud_l g)\dot~=\ud_l\dot g - A_{lr}(v)\ud_rg
\end{equation}
with the matrix
$$A_{lr}(v)=\ud_lv_r-\nu_s\nu_l\ud_rv_s\quad (l,r=1,\dots,n+1).$$
\end{lemma}
%%%%%%%%%%%%%%%%%%%%%%%%%%%%%%%%%%%%%%%%%%%%%%%%
\section{Derivation of the PDE}
We derive the conservation law, which we are going to solve in
this paper. For this we need the following transport theorem on
moving surfaces. A proof can be found in \cite{dziuk1}.
\begin{lemma}\label{TransportTheorem}
\begin{equation}\label{Leibniz}
\frac{d}{dt}\int_\Gamma g \,\dA
= \int_\Gamma \dot g + g\nabla_\Gamma\cdot v\,\dA.
\end{equation}
\end{lemma}
%\vspace{0.5cm}

Let $u(\cdot,t)$ be a scalar quantity, defined on $\Gamma(t)$, which is conserved.
The conservation law which we are going to solve is given in integral form by
\begin{equation}\label{conservation}
\frac{d}{dt}\int_{\gamma(t)}u\,\dA=\int_{\partial \gamma(t)} Q\cdot \mu\,\dS.
\end{equation}
Here $\gamma(t)$ is a portion of $\Gamma(t)$ which moves in time according to
the prescribed velocity $v$. $Q$ is a flux, which we will parametrize later.
Obviously normal parts of $Q$ do not enter the conservation law, because
the conormal $\mu$ on $\partial \gamma$ is a tangent vector. Thus we may
assume that $Q$ is a tangent vector to $\Gamma$. But note, that even if we choose
$Q$ as a vector with a normal part, then $Q\cdot \mu = PQ\cdot \mu$ with the
projection $P_{ij}=\delta_{ij}-\nu_i\nu_j$ ($i,j=1,\dots,n+1$).

We apply integration by parts (\ref{IntParts}) to the right hand side of (\ref{conservation}),
$$\int_{\partial \gamma(t)} Q\cdot \mu\,\dS=\int_{\gamma(t)}\nabla_\Gamma\cdot Q\,\dA
-\int_{\gamma(t)}H Q\cdot\nu\,\dA=\int_{\gamma(t)}\nabla_\Gamma\cdot Q\,\dA.$$
To the left hand side of (\ref{conservation}) we apply the transport theorem from
Lemma \ref{TransportTheorem}. This leads to
$$\frac{d}{dt}\int_{\gamma(t)}u\,\dA=\int_{\gamma(t)}\dot u + u\nabla_\Gamma\cdot v\,\dA.$$
Thus the equation (\ref{conservation}) is equivalent to
\begin{equation*}\label{cons-2}
\int_{\gamma(t)}\dot u + u\nabla_\Gamma\cdot v-\nabla_\Gamma\cdot Q\,\dA=0,
\end{equation*}
and since $\gamma$ is an arbitrary subregion of $\Gamma$, we arrive at the
PDE
\begin{equation*}\label{cons-0}
\dot u + u\nabla_\Gamma\cdot v-\nabla_\Gamma\cdot Q =0.
\end{equation*}

Throughout this paper we assume, that $Q$ has the form
\begin{equation*}\label{Fluss}
Q=-f((x,t),u)\quad (x\in\Gamma(t)),
\end{equation*}
where we assume that
\begin{equation}\label{Fluss-Tang}
f(\cdot,u)\cdot\nu=0\
\end{equation}
for all $u\in \R$ on $\Gamma$.
With this parametrization of the flux the PDE (\ref{cons-0}) reads
\begin{equation}\label{PDE}
\dot u + u\nabla_\Gamma\cdot v+\nabla_\Gamma\cdot (f(\cdot,u))=0\quad
\mbox{ on } G_T.
\end{equation}
By the divergence of $f$ we mean the `total' divergence
\begin{eqnarray*}
\label{PDE-explizit}
\nabla_\Gamma\cdot (f(\cdot,u))
&=&\sum_{j=1}^{n+1}\ud_jf_j(\cdot,u)
+\sum_{j=1}^{n+1}\frac{\partial f_j}{\partial u}(\cdot,u)\ud_ju \\
&=&\nabla_\Gamma\cdot f(\cdot,u)+f_u(\cdot,u)\cdot \nabla_\Gamma u.
\nonumber
\end{eqnarray*}
\begin{remark}
Note, that because of the condition (\ref{Fluss-Tang})
it is in general not possible to choose the flux $f$ independently of $x$ and $t$.
If we start with a flux of the form $Q=\tilde f(u)$ in the law (\ref{conservation}),
then the PDE changes to
$$\dot u + u\nabla_\Gamma\cdot v +\nabla_\Gamma\cdot P\tilde f(u)=0 $$
and we have $f((x,t),u)=P(x,t)\tilde f(u)$ in (\ref{PDE}).
\end{remark}
%%%%%%%%%%%%%%%%%%%%%%%%%%%%%%%%%%%%%%%%%%%%%%
\section{Definition of entropy solutions}
As in the Euclidean case classical solutions of (\ref{pde})
do not exist globally in time in general. Therefore we have to introduce the notion of a weak solution.
\begin{definition}\label{def:weaksol}
A function $u\in L^\infty(G_T)$ is called a {\em weak solution} of \eqref{pde} if
\begin{equation}
\int_0^T \int_{\Gamma} u\dot{\varphi}+f(\cdot,u)\cdot \nabla_{\Gamma} \varphi\,\dA+\int_{\Gamma_0}u_0\varphi(\cdot,0)\,\dA = 0
\label{eq:weaksol}
\end{equation}
for all test functions $\varphi \in C^1(\overline{G_T})$ with $\varphi(\cdot,T)=0$.
\end{definition}
In general weak solutions are not unique. Therefore we select the entropy solution which will be introduced in
Definition \ref{entropy12}. For the motivation of the entropy condition given in (\ref{weak1}),
let us consider the following Lemma.

Here and in the following we assume that $u_{0\varepsilon} \in H^{2,1}\cap L^\infty  $ and
\begin{equation}\label{aw-reg}
\|u_{0\veps}\|_{L^\infty(\Gamma_0)}+\|\nabla_{\Gamma_0}u_{0\veps}\|_{L^1(\Gamma_0)}
+\veps\|\nabla_{\Gamma_0}^2u_{0\veps}\|_{L^1(\Gamma_0)}\leq c_0
\end{equation}
with a constant $c_0$ which is independent of the parameter $0<\veps\leq 1$.
\begin{lemma}\label{ent-con}
Let $f=(f_1,\dots,f_{n+1})$, $q=(q_1,\dots,q_{n+1})$, $\eta \in C^2(\R)$, $\eta'' \ge 0$.
Define $q_l(\cdot,s):=\int_{s_0}^s\eta'(\tau)f_{lu}(\cdot,\tau)d\tau$ for
$l=1,\dots, n+1$ and let $u_0\in L^{\infty}(\Gamma_0)$.
Assume that $u_\varepsilon$ is a smooth solution of
\begin{equation} %\label{eps}
{\dot u}_\varepsilon +u_\varepsilon
\nabla_{\Gamma}\cdot v +\nabla_{\Gamma} \cdot f(\cdot,u_\varepsilon)
-\varepsilon \Delta_{\Gamma}u_\varepsilon =0 \,\, \mbox{on} \,\, G_T,
%\nonumber\label{pde-eps}\\
\quad  u_\varepsilon(\cdot,0)=u_{0\varepsilon} \,\, \mbox{on}\,\, \Gamma_0
\label{initialvalueproblem-eps}.
\end{equation}
If $u_\varepsilon \to u $ a.e. on $G_T$ and $u_{0\varepsilon} \to u_0 $ a.e.
on $\Gamma_0$ for $\varepsilon \to 0$ and $u\in L^1(G_T)$, then $u$ satisfies
the {\em entropy condition}
\begin{equation} \label{weak1}
-\int_{\Gamma_0}
\eta(u_0)\phi(\cdot,0) +\int_0^T \int_{\Gamma_t} \big(-\eta(u)
\dot\phi(\cdot,t) -q(\cdot, u)\cdot \nabla_\Gamma\phi +
\nabla_\Gamma\cdot v(u\eta'(u)-\eta(u))\phi\big)  \le 0
\end{equation}
for all test functions $\phi\in H^1(G_T)$ with $\phi \ge 0$ and
$\phi(\cdot, T)=0$.
\end{lemma}
\begin{proof} Let $\eta$ and $q$ be defined as above.
We multiply (\ref{pde}) by $\eta'(u_\varepsilon)$ and
obtain
\begin{equation}
%\label{etau}
\dot u_\veps \eta'(u_\veps) +u_\veps \nabla_{\Gamma}\cdot
v \eta'(u_\veps) +\nabla_{\Gamma} \cdot f(\cdot,u_\veps) \eta'(u_\veps) -\varepsilon
\Delta_{\Gamma}u_\veps
\eta'(u_\veps) =0 \,\,\, \mbox{on} \,\, G_T\label{pde1}.
%\nonumber
\end{equation}
This implies
\begin{equation}
\dot{\eta(u_\veps)} +u_\veps \nabla_{\Gamma}\cdot v
\eta'(u_\veps) + f_l'(\cdot,u_\veps)D_lu_\veps \eta'(u_\veps) -\varepsilon \Delta_{\Gamma} \eta(u_\veps)
+ \varepsilon \eta''(u_\veps)(D_lu_\veps)^2=0 \,\,\,\,\,\,\, \mbox{on} \,\, G_T\nonumber\label{pde2}.
\end{equation}
We multiply by a smooth test function $\phi$ such that
$\phi\ge 0, \, \phi(\cdot,T) =0$ and integrate. This gives
\begin{equation}
\int_0^T\int_{\Gamma_t}\dot{ \eta(u_\veps)} \phi + u_\veps \nabla_\Gamma \cdot v
\eta '(u_\veps)\phi +f_l'D_lu\eta'(u_\veps)\phi -\varepsilon \Delta_\Gamma \eta(u_\veps) \phi
+ \varepsilon \eta''(u_\veps)
(D_lu_\veps)^2 \phi =0.\label{etapunkt}
\end{equation}
Since
\bea
R&:=& -\int_{\Gamma_0} \eta (u_\veps(\cdot,0))\phi(\cdot,0)=
\int_0^T \frac{d}{dt}\int_{\Gamma_t} \eta
(u_\veps(\cdot,t))\phi(\cdot,t)\nonumber\\
&=& \int_0^T \int_{\Gamma_t} \big(\eta(u_\veps(\cdot,t))\phi(\cdot,t)\big)^{\cdot}
+ \int_0^T \int_{\Gamma_t}\eta(u_\veps)\phi \nabla_\Gamma  \cdot v \nonumber\\
&=& \int_0^T \int_{\Gamma_t} \big(\eta(u_\veps(\cdot,t))^{\cdot}\phi(\cdot,t)
+ \eta(u_\veps(\cdot,t))\dot{\phi}(\cdot,t)\big)
+ \int_0^T \int_{\Gamma_t}\eta(u_\veps)\phi \nabla_\Gamma  \cdot v \nonumber,
\nonumber
\eea
 we obtain from (\ref{etapunkt}):
\begin{eqnarray*}
&&R-\int_0^T \int_{\Gamma_t} \eta(u_\veps(\cdot,t))\dot\phi(\cdot,t)\\
&&- \int_0^T \int_{\Gamma_t}\big(\eta(u_\veps)\phi \nabla_\Gamma \cdot v
-u_\veps\nabla_\Gamma \cdot v\eta'(u_\veps) \phi
 -f_{lu}D_lu_\veps\eta'(u_\veps)\phi +\varepsilon \eta(u_\veps) \Delta_\Gamma \phi \big) \nonumber\\
&=&R-  \int_0^T \int_{\Gamma_t} \eta (u_\veps(\cdot,t))\dot\phi(\cdot,t)\nonumber\\
&&- \int_0^T \int_{\Gamma_t}\big(\eta(u_\veps)\phi \nabla_\Gamma  \cdot v -u_\veps\nabla_\Gamma  \cdot v\eta'(u_\veps) \phi
+q(u_\veps)\cdot \nabla_{\Gamma} \phi +\varepsilon \eta(u_\veps) \Delta_\Gamma \phi \big)
\le 0 . \label{tangentf}
\end{eqnarray*}
This means that we have the inequality
\bea
&&-\int_{\Gamma_0} \eta (u_{0\varepsilon})\phi(\cdot,0)
-\int_0^T \int_{\Gamma_t} \eta (u_\varepsilon(\cdot,s))\dot\phi(\cdot,s)\\
&&\qquad - \int_0^T \int_{\Gamma_t}\big(\eta(u_\varepsilon)\phi \nabla_\Gamma\cdot v
 -u_\varepsilon\nabla_\Gamma  \cdot v\eta'(u_\varepsilon) \phi
 +q(u_\varepsilon) \cdot \nabla_{\Gamma} \phi +\varepsilon \eta(u_\varepsilon)
 \Delta_\Gamma \phi \big) \le 0, \label{tangentf1}\nonumber
\eea
and for $\varepsilon \to 0 $ we obtain in the limit
 \bea &&-\int_{\Gamma_0} \eta (u_{0})\phi(\cdot,0)
 - \int_0^T \int_{\Gamma_t} \eta (u(\cdot,s))\dot\phi(\cdot,s)\\
 &&\qquad - \int_0^T \int_{\Gamma_t}\big(\eta(u)\phi \nabla_\Gamma \cdot  v
 -u \nabla_\Gamma  \cdot v\eta'(u) \phi
 +q(u)\cdot \nabla_{\Gamma} \phi  \big) \le 0. \label{tangentf2}
 \nonumber
 \eea
This finally proves the Lemma.
\end{proof}

Now we use the property (\ref{weak1}) for the definition of an entropy solution.
\begin{definition} \label{entropy12}
Let $\eta, q_l$ and $u_0$ be as in Lemma \ref{ent-con}. Then $u \in L^{\infty}(G_T)$
is an {\em entropy solution} (admissible weak solution)
of (\ref{pde}) if
\begin{equation}
-\int_{\Gamma_0} \eta(u_0)\phi(\cdot,0)
+\int_0^T \int_{\Gamma_t} \big(-\eta(u)
\dot\phi(\cdot,t) -q(\cdot, u) \cdot \nabla_\Gamma\phi + \phi
\nabla_\Gamma\cdot v(u\eta'(u)-\eta(u))\big)  \le 0
\label{entropy11}
\end{equation}
holds for all test functions $\phi\in H^1(G_T)$ with $\phi \ge 0$ and
$\phi(\cdot,T)=0$
and for all $\eta$ and $q$ with the properties,
mentioned above.
\end{definition}
\begin{remark} If we choose $\eta(u)=u$ in Definition \ref{entropy12}, then this
implies $q(\cdot,u)=f(\cdot,u)+const$, $u\eta'(u)-\eta(u)=0$
and that $u$ is a weak solution of (\ref{pde}).
\end{remark}

The following definition of Kruzkov entropy solutions is equivalent to Definition \ref{entropy12}.
\begin{definition}\label{def:kruzkoventr}
A function $u \in L^\infty(G_T)$ is called
{\em Kruzkov entropy solution} of \eqref{pde} if
\begin{eqnarray}
\nonumber
\int_0^T\int_{\Gamma} |u-k|\dot{\varphi}-\mbox{sign}(u-k)\,k\,\nabla_{\Gamma}\cdot v\,\varphi
+\mbox{sign}(u-k) (f(\cdot,u)-f(\cdot,k))\nabla_{\Gamma}\varphi \\
+\int_{\Gamma_0} |u_0-k| \varphi(\cdot,0) \geq 0
\label{eq:krukoventr}
\end{eqnarray}
for all $k\in\R$ and
all test functions $\varphi \in C^1(\overline{G_T})$ with $\varphi \geq 0$ and $\varphi(\cdot,T)=0$.
\end{definition}

\begin{remark} \label{remark-entropy}
An entropy solution is a Kruzkov entropy solution. This can be seen by a regularization of the Kruzkov entropy - entropy flux pair.
See for example \cite{Kro97}.
%\todo{nochmal nachchecken}
\end{remark}

\section{The regularized problem}
In order to solve the conservation law (\ref{pde})
we solve the initial value problem (\ref{initialvalueproblem-eps})
and  consider $u_\varepsilon$ for $\varepsilon \to 0$.
For technical reasons let us consider the following regularized PDE
\begin{equation}\label{regpde}
\dot u_\varepsilon + u_\varepsilon\nabla_\Gamma\cdot v +\nabla_\Gamma\cdot f(\cdot,u_\varepsilon)
-\veps \nabla_\Gamma \cdot \left(B \nabla_\Gamma u_\varepsilon\right)=0
\end{equation}
on ${G}_T$
with initial data $u_\varepsilon(\cdot,0)=u_{0\varepsilon}$ on $\Gamma_0$
with $u_{0\varepsilon} \to u_0$ a.e. on $\Gamma_0$ and (\ref{aw-reg}).
Here $B=B(x,t)$ is a symmetric diffusion matrix which maps
the tangent space of $\Gamma(t)$ into the tangent space at the point
$x\in\Gamma(t)$, so that we have
$B\nu=0$ and $\nu^* B=0$. Assume also that $B$ is positive definite on the
tangent space. Similarly as in Lemma \ref{ent-con} it can be shown
that $u$ is an entropy solution if $u_\varepsilon \to u$ for $\veps\to 0$.
In the proofs of the following Section we will use the
fact that
\begin{equation}\label{BP}
BP=PB=B.
\end{equation}

%We suppress the dependency of $u$ on the parameter $\veps$ for
%better readability of the following.
The main purpose of
the next Section is to prove a priori bounds for $u_\varepsilon$ which are
independent of $\veps$.

\section{A Priori estimates for the regularized problem}
In this Section we replace $u_\veps$ by $u$ for
better readability. The aim of this Section is the derivation of a priori
estimates which are independent of $\veps$.
The initital value problem
\begin{equation}\label{reg-pde}
\dot u + u\nabla_\Gamma\cdot v +\nabla_\Gamma\cdot f(\cdot,u)
-\veps \nabla_\Gamma \cdot \left(B \nabla_\Gamma u\right)=0,
\quad u(\cdot,0)=u_{0\veps}
\end{equation}
has a unique smooth solution. This is shown by dovetailing the
cut-off technique of Kruzkov with the Galerkin ansatz from \cite{dziuk1}.
The proof is quite straight forward and so we omit the details here.
The proof of smoothness of the weak solution found by this method is a purely
local argument.

\subsection{Estimate of the solution}
We prove that the solution $u$ of the regularized
parabolic initital value problem
(\ref{reg-pde})
is bounded in the $L^\infty$-norm in space
and time independently of the parameter $\veps$.
\begin{lemma}\label{funk-absch-Lemma}
Let $u$ be the solution of (\ref{reg-pde}). Then
\begin{equation}\label{funk-absch}
\sup_{t\in (0,T)}\|u(\cdot,t)\|_{L^\infty(\Gamma(t))}\leq c
\end{equation}
with a constant $c$ which is independent of $\veps$ but depends on
the data of the problem including the final time $T$ and $c_0$ from (\ref{aw-reg}).
\end{lemma}
\begin{proof}
This estimate is a consequence of
the maximum principle for parabolic PDEs. Because of the
unusual setting here, we give a proof. The PDE (\ref{regpde}) can
be written in a weak form. Note that the nonlinearity $f((x,t),u)$ is
tangentially divergence free with respect to the $x$-variable.
This is crucial here.
We begin by transforming $u$:
$$w(x,t)=e^{-\lambda t}u(x,t),\quad x\in\Gamma(t),$$
where we set $\lambda=\sup_{(0,T)}\|\left(\nabla_\Gamma\cdot v\right)_{-}\|_{L^\infty(\Gamma)}$
with $\left(\nabla_\Gamma\cdot v\right)_-=\min\{\nabla_\Gamma\cdot v,0\}$.
We set
$g(\cdot,w)=e^{-\lambda t}f(\cdot,e^{\lambda t}w)\psi(w)$
with a function $\psi\in C^{1}_0(\R)$ which
satisfies $\psi(w)=1$ for $|w|\le M_\veps$ and $\psi(w)=0$ for $|w|\ge M_\veps+1$.
We use $M_\veps= \|u\|_{L^\infty(G_T)}$.  Then
\begin{equation}\label{g-reg}
\left|\frac{\partial g}{\partial w}(\cdot,w)\right| \leq c(M_\veps).
\end{equation}
Because of $|w|=e^{-\lambda t}|u|\leq M_\veps$ we then have
\begin{eqnarray}\label{regpdeweak-w}
\int_\Gamma \dot w\varphi+\int_\Gamma w\varphi\left(\lambda +\nabla_\Gamma\cdot v\right)
-\int_\Gamma g(\cdot,w)\cdot\nabla_\Gamma\varphi
+\varepsilon \int_\Gamma B\nabla_\Gamma w\cdot \nabla_\Gamma\varphi = 0
\end{eqnarray}
for every $\varphi$. If we choose $\varphi=(w-M)_+=\max\{w-M,0\}$
with $M=c_0\geq \|u_{0\veps}\|_{L^\infty(\Gamma_0)}$ in this equation, then we arrive at
\begin{eqnarray*}
\lefteqn{\frac 12 \int_\Gamma \left((w-M)_+^2\right){\dot~}
+\varepsilon \int_\Gamma B\nabla_\Gamma (w-M)_+\cdot\nabla_\Gamma (w-M)_+}\\
&&=\int_\Gamma (g(\cdot,w)-g(\cdot,M))\cdot \nabla_\Gamma(w-M)_+
-\int_\Gamma w(w-M)_+(\lambda +\nabla_\Gamma\cdot v).
\end{eqnarray*}
Here we have used that $g(\cdot,M)\cdot\nu=0$.
For the right hand side of this equation we observe that
for our choice of $\lambda$
$$w(w-M)_+(\lambda +\nabla_\Gamma\cdot v)\ge 0,$$
and that
$$|g(\cdot,w)-g(\cdot,M)|\le c(M_\veps) |w-M|.$$
We use these two estimates together with the ellipticity condition and get
with a positive constant $c$ the estimate
\begin{eqnarray*}
\lefteqn{
\frac 12 \frac{d}{dt}\int_\Gamma (w-M)_+^2
+\varepsilon c \int_\Gamma |\nabla_\Gamma (w-M)_+|^2 }\\
&&\le \frac 12 \|(\nabla_\Gamma \cdot v)_+\|_{L^\infty(\Gamma)}\int_\Gamma (w-M)^2_+
+c(M_\veps)\int_\Gamma (w-M)_+|\nabla_\Gamma (w-M)_+|.
\end{eqnarray*}
Here we also have used the transport theorem from Lemma \ref{TransportTheorem}.
From the previous estimate we infer with Young's inequality that
\begin{equation*}
\frac{d}{dt} \int_\Gamma (w-M)_+^2 +\varepsilon c\int_\Gamma |\nabla_\Gamma (w-M)_+|^2
\leq c(M_\veps,\varepsilon)\int_\Gamma (w-M)_+^2.
\end{equation*}
This implies for the nonnegative function $\phi(t)=\int_{\Gamma(t)} (w(\cdot,t)-M)_+^2$ the
inequality
$$\phi^\prime(t)\le c(M_\veps,\varepsilon)\phi(t).$$
Because of $\phi(0)=0$ we then obtain with a Gronwall argument that $\phi(t)=0$. But
this says that $(w-M)_+=0$ almost everywhere which implies
$w\le M$ or
$$u(\cdot,t)\le c(T,M) \mbox{ on } \Gamma(t),$$
and the constant $c(T,M)$ does not depend on $\veps$.
The estimate from below follows similarly.
\end{proof}
%%%%%%%%%%%%%%%%%%%%%%%%%%%%%%%%%%%%%
\subsection{Estimate of the spatial gradient}
\begin{lemma}\label{gradient-absch}
Assume that $u$ solves the regularized PDE (\ref{reg-pde}). Then
\begin{equation}%\label{dotpde}
\sup_{(0,T)}\int_\Gamma |\nabla_\Gamma u|\leq c
\end{equation}
with a constant $c$ which does not depend on $\veps$.
\end{lemma}
\begin{proof}
Set $w_i=\ud_iu$ and $w=(w_1,\dots,w_{n+1})$.
We take the derivative $\ud_i$ of the regularized PDE (\ref{reg-pde}),
\begin{equation*}
\ud_i\dot u + \ud_i(u\nabla_\Gamma \cdot v) +\ud_i\nabla_\Gamma\cdot(f(\cdot,u))
-\veps \ud_i\nabla_\Gamma \cdot \left(B \nabla_\Gamma u\right)=0
\end{equation*}
and treat the terms separately.
With the equation (\ref{dotgrad}) we get
\begin{equation}\label{di-1}
\ud_i\dot u= (\ud_iu)\dot~ + A_{ir}(v)\ud_r u=\dot w_i+A_{ir}(v)w_r.
\end{equation}
Clearly
\begin{equation}\label{di-2}
\ud_i(u\nabla_\Gamma\cdot v)=w_i\nabla_\Gamma\cdot v+u\ud_i\nabla_\Gamma\cdot v.
\end{equation}
For the nonlinear term we get with the use of (\ref{vertauschung})
\begin{eqnarray}
\lefteqn{
\ud_i\nabla_\Gamma\cdot(f(\cdot,u))}\label{di-3}\\
&&= \ud_i(\ud_kf_k(\cdot,u)+f_{ku}(\cdot,u)\ud_ku)\nonumber\\
&&=\ud_i\ud_kf_k(\cdot,u)+\ud_kf_{ku}(\cdot,u)w_i+\ud_if_{ku}(\cdot,u)w_k\nonumber\\
&&+f_{kuu}(\cdot,u)w_iw_k+f_{ku}(\cdot,u)\left(\ud_kw_i+\H_{kl}w_l\nu_i
-\H_{il}w_l\nu_k\right).\nonumber
\end{eqnarray}
With (\ref{vertauschung}) the second order term can be rewritten as follows:
\begin{eqnarray}
\lefteqn{ \ud_i\nabla_\Gamma \cdot(B\nabla_\Gamma u)
=\ud_i\ud_k(B_{km}\ud_mu)}\label{di-4}\\
&&=\ud_k\ud_i(B_{km}\ud_mu)
+\H_{kl}\ud_l(B_{km}\ud_mu)\nu_i-\H_{il}\ud_l(B_{km}\ud_mu)\nu_k\nonumber\\
&&=\ud_k(\ud_iB_{km}\ud_mu)+\ud_k(B_{km}(\ud_m\ud_iu+\H_{ml}\ud_lu\nu_i-\H_{il}\ud_lu\nu_m))\nonumber\\
&&+\H_{kl}\ud_l(B_{km}\ud_mu)\nu_i-\H_{il}\ud_l(B_{km}\ud_mu)\nu_k\nonumber\\
&&=\ud_k(\ud_iB_{km} w_m)+
\ud_k(B_{km}(\ud_mw_i+\H_{ml}w_l\nu_i-\H_{il}w_l\nu_m))\nonumber\\
&&+\H_{kl}\ud_l(B_{km}w_m)\nu_i-\H_{il}\ud_l(B_{km}w_m)\nu_k\nonumber\\
&&=\ud_k(B_{km}\ud_mw_i)+\ud_k(\ud_iB_{km}w_m)
+\ud_k(B_{km}(\H_{ml}w_l\nu_i-\H_{il}w_l\nu_m))\nonumber\\
&&+\H_{kl}\ud_l(B_{km}w_m)\nu_i-\H_{il}\ud_l(B_{km}w_m)\nu_k\nonumber\\
&&=\ud_k(B_{km}\ud_mw_i)+\ud_k(\ud_iB_{km}w_m)
+\ud_k(B_{km}\H_{ml}w_l)\nu_i\nonumber\\
&&+B_{km}\H_{ml}\H_{ik}w_l +\H_{kl}\ud_l(B_{km}w_m)\nu_i
+\H_{il}\H_{lk}B_{km}w_m.
\nonumber
\end{eqnarray}
For the last term we have used $\H_{il}\ud_l(B_{km}w_m)\nu_k=-\H_{il}\H_{lk}B_{km}w_m$.
We now collect the intermediate results (\ref{di-1}), (\ref{di-2}), (\ref{di-3})
and (\ref{di-4}) to arrive at the following PDE for $w_i=\ud_iu$.
\begin{eqnarray*}\label{dipde}
\lefteqn{\dot w_i+A_{ir}(v)w_r+w_i\nabla_\Gamma\cdot v+u\ud_i\nabla_\Gamma\cdot v
+\ud_i\ud_kf_k(\cdot,u)+\ud_kf_{ku}(\cdot,u)w_i }\\
&&+\ud_if_{ku}(\cdot,u)w_k
+f_{kuu}(\cdot,u)w_iw_k+f_{ku}(\cdot,u)\left(\ud_kw_i+\H_{kl}w_l\nu_i
-\H_{il}w_l\nu_k\right)\\
&&-\veps\left(\ud_k(B_{km}\ud_mw_i)+\ud_k(\ud_iB_{km}w_m)
+\ud_k(B_{km}\H_{ml}w_l)\nu_i\right)\\
&&-\veps\left(\H_{ik}\H_{ml}B_{km}w_l+\H_{kl}\ud_l(B_{km}w_m)\nu_i
+\H_{il}\H_{lk}B_{km}w_m\right)=0.
\end{eqnarray*}
We multiply this equation by $\frac{w_i}{|w|}$,
sum with respect to $i$ from $1$ to $n+1$, use the fact that $w$ is a tangent
vector, and get
\begin{eqnarray} \label{wdgl}
\lefteqn{ |w|\dot~+|w|\nabla_\Gamma\cdot v + |w|\ud_kf_{ku}(\cdot,u)+f_{kuu}(\cdot,u)|w|w_k
+f_{ku}(\cdot,u)\frac{w_i}{|w|}\ud_kw_i }\\
&&+\frac{w_i}{|w|}\ud_i\ud_kf_k(\cdot,u)\nonumber
+\ud_if_{ku}(\cdot,u)\frac{w_i}{|w|}w_k
-\veps \frac{w_i}{|w|}\ud_k(B_{km}\ud_mw_i)\\
\nonumber
&&=-\ud_iv_r \frac{w_i}{|w|}w_r+\veps\frac{w_i}{|w|}\ud_k(\ud_iB_{km}w_m)\\
&& +\veps B_{km}\H_{ml}\H_{ik}w_l\frac{w_i}{|w|}+\veps B_{km}\H_{il}\H_{lk}w_m\frac{w_i}{|w|}
-u\frac{w_i}{|w|}\ud_i\nabla_\Gamma\cdot v.  \nonumber
\end{eqnarray}
Here we also used that $f_{ku}(\cdot,u)\nu_k=0$ and $w_i\nu_i=0$.
We now observe that
\begin{equation}\label{w-1}
%\ud_k(f_{ku}(\cdot,u)|w|)=
\ud_kf_{ku}(\cdot,u)|w|+f_{kuu}(\cdot,u)w_k|w|+f_{ku}(\cdot,u)\frac{w_i}{|w|}\ud_kw_i
=\ud_k(f_{ku}(\cdot,u)|w|).
\end{equation}
The result of Lemma \ref{funk-absch-Lemma} allows us to estimate some terms in (\ref{wdgl}).
\begin{eqnarray}\label{w-2}
\lefteqn{|w|\dot~ + |w|\nabla_\Gamma\cdot v +\nabla_\Gamma\cdot(f_u(\cdot,u)|w|)\
-\veps \frac{w_i}{|w|}\ud_k(B_{km}\ud_mw_i) }\\
&&\leq c_1 + c_2 |w| + c_3 |u|
+\veps \frac{w_i}{|w|}\ud_k(\ud_iB_{km}w_m)
\nonumber
\end{eqnarray}
We rewrite the second order term on the left hand side of this equation
(integrated over $\Gamma$).
\begin{eqnarray*}
\lefteqn{
-\veps\int_\Gamma \frac{w_i}{|w|}\ud_k\left(B_{km}\ud_mw_i\right) }\\
&&=-\veps\int_\Gamma\frac{w_i}{|w|}H\nu_kB_{km}\ud_mw_i
+\veps\int_\Gamma\ud_k\left(\frac{w_i}{|w|}\right)B_{km}\ud_mw_i\\
&&=\veps\int_\Gamma\frac{1}{|w|}Q_{il}B_{km}\ud_kw_l\ud_mw_i,
\end{eqnarray*}
where we have set
$$Q_{il}=\delta_{il}-\frac{w_iw_l}{|w|^2},\quad i,l=1,\dots,n+1.$$
We now estimate the last term on the right hand side of (\ref{w-2}) integrated over $\Gamma$.
\begin{eqnarray*}
\lefteqn{
\veps \int_\Gamma \frac{w_i}{|w|}\ud_k\left(\ud_iB_{km}w_m\right) }\\
&& =\veps\int_\Gamma\frac{w_i}{|w|}H\nu_k\ud_iB_{km}w_m
-\veps\int_\Gamma \frac{1}{|w|}Q_{il}\ud_kw_l\ud_iB_{km}w_m\\
&&=-\veps\int_\Gamma \frac{w_i}{|w|}HB_{km}\H_{ik}w_m
-\veps\int_\Gamma\frac{1}{|w|}Q_{il}\ud_iB_{km}w_m\ud_kw_l\\
&&\leq c_4 \veps \int_\Gamma |w|
-\veps\int_\Gamma\frac{1}{|w|}Q_{il}\ud_iB_{km}w_m\ud_kw_l
\end{eqnarray*}
Since the matrix $Q$ is positive semidefinite and since $B$ is
positive definite on the tangent space we can estimate the last
term on the right hand side. We use the abbreviations
$$z_{rl}=\ud_rw_l,\quad \zeta_{is}=B^{-1}_{sk}P_{kp}\ud_iB_{pm}w_m.$$
For any $\delta > 0$ we then have
\begin{eqnarray*}
\lefteqn{
\left|Q_{il}\ud_iB_{km}\ud_kw_l w_m\right|
=\left|Q_{il}B_{rs}z_{rl}\zeta_{is}\right| }\\
&&\leq \sqrt{Q_{il}B_{rs}z_{rl}z_{is}}\sqrt{Q_{il}B_{rs}\zeta_{is}\zeta_{rl}}\\
&&\leq \frac{\delta}{2}Q_{il}B_{rs}z_{rl}z_{is}
+\frac{1}{2\delta}Q_{il}B_{rs}\zeta_{is}\zeta_{rl}\\
&&\leq  \frac{\delta}{2}Q_{il}B_{rs}\ud_rw_l\ud_iw_s
+c(\delta)|w|^2.
\end{eqnarray*}

The first equation can be obtained as follows:

\begin{eqnarray*}
Q_{il}B_{rs}z_{rl}\zeta_{is}&=&Q_{il}B_{rs}\ud_rw_lB^{-1}_{sk}P_{kp}\ud_iB_{pm} w_m=Q_{il}\ud_rw_lP_{kp}\ud_iB_{pm} w_m \delta_{rk}\\
&=& Q_{il}\ud_rw_lP_{rp}\ud_i B_{pm} w_m =Q_{il}\ud_rw_l(\delta_{rp}-\nu_r\nu_p)\ud_i B_{pm} w_m\\
&=&Q_{il}\ud_pw_l\ud_i B_{pm} w_m.
\end{eqnarray*}

Now because of
$$\ud_iw_s=\ud_sw_i+\H_{sp}w_p\nu_i-\H_{ip}w_p\nu_s$$
we get
\begin{eqnarray*}
Q_{il}B_{rs}\ud_rw_l\ud_iw_s=
Q_{il}B_{rs}\ud_rw_l\ud_sw_i+Q_{il}B_{rs}\H_{sp}\nu_i\ud_rw_l w_p
\end{eqnarray*}
and because of $Q_{il}\nu_i=\nu_l$,
the second term on the right hand side can be rewritten as follows:
\begin{eqnarray*}
Q_{il}B_{rs}\H_{sp}\nu_i\ud_r w_l w_p
&=&\nu_lB_{rs}\H_{sp}\nu_i\ud_r w_l w_p \\
&=&B_{rs}\H_{sp}\nu_i\left(\ud_r (w_l\nu_l)-w_l\H_{rl}\right)w_p\\
&=&-B_{rs}\H_{sp}\H_{rl}\nu_i w_l w_p.
\end{eqnarray*}
Collecting the previous estimates we arrive at
\begin{eqnarray*}
\veps\int_\Gamma\frac{w_i}{|w|}\ud_k(\ud_iB_{km}w_m)
\leq c \veps\int_\Gamma|w|
+\frac{\veps}{2}\int_\Gamma \frac{1}{|w|}Q_{il}B_{km}\ud_kw_l\ud_mw_i.
\end{eqnarray*}
We integrate (\ref{w-2}) over $\Gamma$ and finally get
\begin{eqnarray}
\lefteqn{
\int_\Gamma |w|\dot~ + |w|\nabla_\Gamma\cdot v
+\int_\Gamma \nabla_\Gamma\cdot\left(f_u(\cdot,u)|w|\right) }\\
&&\leq - \frac{\veps}{2}\int_\Gamma \frac{1}{|w|}
Q_{il}B_{km}\ud_kw_l\ud_mw_i
+c\int_\Gamma|w| \nonumber\\
&&\leq c \int_\Gamma |w|+c\int_\Gamma|u|+c\int_\Gamma |u| + c,
\nonumber
\end{eqnarray}
where we again have used the fact that $\Gamma$ is compact.
Since
\begin{equation*}
\int_\Gamma\nabla_\Gamma\cdot(f(\cdot,u)|w|)
=\int_\Gamma H\nu\cdot f(\cdot,u)|w|=0,
\end{equation*}
we finally get the estimate
\begin{equation*}\label{w-3}
\frac{d}{dt}\int_\Gamma|w|
\leq c\int_\Gamma|w|+c.
\end{equation*}
Here we have used the bound (\ref{funk-absch})
for $u$. In summary we have proved that
$$\sup_{(0,T)}\int_\Gamma|w|\leq c$$
independently of $\veps$. In the last step we used (\ref{aw-reg}). Lemma \ref{gradient-absch} is proved.
\end{proof}
%%%%%%%%%%%%%%%%%%%%%%%%%%%%%%%%%%%%%%%%%%%%%%%%%%%%%%
\section{Estimate of the time derivative}
Assume that the matrix $B=\left(B_{ik}\right)_{i,k=1,\dots,n+1}$
satisfies
\begin{equation}\label{ode}
\dot B = B A(v)+A(v)^* B+\lambda B,\, B(\cdot,0)=B_0
\end{equation}
where $\lambda>0$ is a constant and $B_0$ is a symmetric and positive definite $(n+1)\times (n+1)$ matrix.
\begin{lemma}\label{BMatrixLemma}
There is a symmetric and positive definite matrix $B$ which solves (\ref{ode}).
\end{lemma}
\begin{proof}
The existence of $B$ follows easily from ODE theory, since (\ref{ode}) is a linear transport equation.

Let us show, that $B$ is symmetric. The transposed matrix solves the
ODE
$$\dot B^*=A(v)^*B^*+B^*A(v)+\lambda B^*.$$
If we subtract this equation from (\ref{ode}), we get
\begin{equation*}
(B-B^*)^{\dot{~}}=(B-B^*)A(v)+A(v)^*(B-B^*)+\lambda (B-B^*)
\end{equation*}
Since $(B-B^*)(\cdot,0)=0$ by assumption, we have that $B-B^*=0$ for all times.

The coercivity is seen as follows.
\begin{eqnarray*}
\left(e^{-\lambda t}B\right)^{\dot{}}=e^{-\lambda t}\left(\dot B-\lambda B\right)
=e^{-\lambda t}\left(BA(v)+A(v)^*B\right).
\end{eqnarray*}
Thus
\begin{eqnarray*}
B(\cdot,t)=e^{\lambda t}\left(B_0 + \int_0^t e^{-\lambda s}\left(BA(v)+A(v)^*B\right)\,ds\right)
\end{eqnarray*}
and from this we get for $\xi\in\R^{n+1}$ with the coercivity of $B_0$,
$B_0\xi\cdot\xi\geq d_0|\xi|^2$ ($d_0>0$), and the smoothness
of $A(v)$ and $B$:
\begin{eqnarray*}
B(\cdot,t)\xi\cdot\xi\geq
e^{\lambda t}B_0\xi\cdot\xi-ce^{\lambda t}|\xi|^2\int_0^te^{-\lambda s}\,ds
\geq \left(e^{\lambda t}(d_0-\frac{c}{\lambda})+\frac{c}{\lambda}\right)|\xi|^2.
\end{eqnarray*}
So, B is positive definite if we choose $\lambda$ big enough.
\end{proof}

\begin{lemma} \label{udot-absch} Assume that $u$ solves the regularized PDE (\ref{regpde}).
Then
\begin{equation}\label{dot-absch}
\sup_{(0,T)}\int_\Gamma|\dot u|\leq c
\end{equation}
with a constant $c$ which does not depend on $\veps$.
\end{lemma}
\begin{proof}
We take the material derivative of (\ref{reg-pde}) and set $z=\dot u$.
\begin{eqnarray*}
\ddot u + \dot u \nabla_\Gamma\cdot v +  u \left(\nabla_\Gamma\cdot v\right)\dot~
+\left(\nabla_\Gamma\cdot (f(\cdot,u))\right)\dot~
- \veps \left(\nabla_\Gamma\cdot (B\nabla_\Gamma u)\right)\dot~=0
\end{eqnarray*}
We use (\ref{dotgrad}) and treat the terms separately.
Clearly
\begin{equation}\label{z-1}
\ddot u + \dot u\nabla_\Gamma\cdot v = \dot z + z \nabla_\Gamma\cdot v.
\end{equation}
For the nonlinear terms we get
\begin{eqnarray}\label{z-2}
\lefteqn{
\label{z-2}
\left(\nabla_\Gamma\cdot f(\cdot,u))\right)\dot~
=\left(\ud_lf_l(\cdot,u)+f_{lu}(\cdot,u)\ud_lu\right)\dot~ }\\
\nonumber
&&=\ud_l\dot f_l(\cdot,u)-A_{lr}(v)\ud_rf_l(\cdot,u)+\dot f_{lu}(\cdot,u)\ud_lu
+f_{luu}(\cdot,u)z\ud_lu\\
\nonumber
&&\quad+f_{lu}(\cdot,u)\left(\ud_l z-A_{lr}(v)\ud_ru\right).
\end{eqnarray}
We calculate the material derivative of the second order term.
\begin{eqnarray*}
\lefteqn{
\left(\nabla_\Gamma\cdot\left(B\nabla_\Gamma u\right)\right)\dot~
=\left(\ud_k\left(B_{kl}\ud_lu\right)\right)\dot~ }\\
&&=\ud_k\left(B_{kl}\ud_lu\right)\dot~-A_{kr}(v)\ud_r\left(B_{kl}\ud_lu\right)\\
&&=\ud_k\left(B_{kl}\left(\ud_lz-A_{lr}(v)\ud_r u\right)
+\dot B_{kl}\ud_l u\right) -A_{kr}(v)\ud_r\left(B_{kl}\ud_lu\right)\\
&&=\nabla_\Gamma\cdot\left(B\nabla_\Gamma z\right)
-\ud_k\left( \left(B_{kr}A_{rl}(v)-\dot B_{kl}+A_{rk}(v)B_{rl}\right)\ud_lu\right)\\
&&\quad+B_{rl}\ud_lu\ud_k(A_{rk}(v))
\end{eqnarray*}
Now choose the matrix $B$ such that it satisfies (\ref{ode}). Then
\begin{equation}\label{z-3}
\left(\nabla_\Gamma\cdot\left(B\nabla_\Gamma u\right)\right)\dot~
=\nabla_\Gamma\cdot\left(B\nabla_\Gamma z\right)
+\lambda\ud_k(B_{kl}\ud_lu)+B_{rl}\ud_lu\ud_k(A_{rk}(v))
\end{equation}
We collect the terms (\ref{z-1}), (\ref{z-2}) and (\ref{z-3})
to get
\begin{eqnarray}
\lefteqn{\label{z-4}
\dot z + z\nabla_\Gamma\cdot v+u(\nabla_\Gamma\cdot v)^{\cdot}
+\ud_l\dot f_l(\cdot,u)-A_{lr}(v)\ud_rf_l(\cdot,u)
+\dot f_{lu}(\cdot,u)\ud_lu }\\
\nonumber
&&\quad+f_{luu}(\cdot,u)z\ud_lu
+f_{lu}(\cdot,u)\ud_lz-A_{lr}(v)f_{lu}(\cdot,u)\ud_ru\\
\nonumber
&&\quad-\veps \nabla_\Gamma\cdot\left(B\nabla_\Gamma z\right)
-\veps B_{rl}\ud_lu\ud_k(A_{kr}(v))
-\lambda\left(z+u\nabla_\Gamma\cdot v + \nabla_\Gamma\cdot f(\cdot,u)\right)
=0.
\end{eqnarray}
We observe that similarly as in (\ref{w-1}) we have
\begin{equation*}
\frac{z}{|z|}\left(f_{lu}(\cdot,u)\ud_lz+zf_{luu}(\cdot,u)\ud_lu\right)
=\nabla_\Gamma\cdot\left(f_u(\cdot,u)|z|\right)-\ud_lf_{lu}(\cdot,u)|z|.
\end{equation*}
Multiplying (\ref{z-4}) with $\mbox{sign} z$ we get (correctly: use
$\frac{z}{\sqrt{\delta^2+z^2}}$ for $\delta\to 0$ etc.)
\begin{eqnarray*}
\lefteqn{
|z|\dot~ +|z|\nabla_\Gamma\cdot v +\nabla_\Gamma\cdot(f_u(\cdot,u)|z|) }\\
&&\leq \veps ~\mbox{sign} z~ \nabla_\Gamma\cdot(B\nabla_\Gamma z)
+c_1 + c_2|\nabla_\Gamma u| +c_3|z|.
\end{eqnarray*}
Here we have used the boundedness of $u$ uniformly with respect
to the regularization parameter $0<\veps\leq 1$.
With the same arguments as in the proof of Lemma \ref{gradient-absch}
after integration over $\Gamma$ we get the inequality
\begin{equation*}
\frac{d}{dt}\int_\Gamma |z|\leq c_1+ c_2\int_\Gamma |\nabla_\Gamma u|+c_3\int_\Gamma|z|,
\end{equation*}
and Lemma \ref{gradient-absch} implies
\begin{equation*}
\sup_{(0,T)}\int_\Gamma |z|\leq c
\end{equation*}
with a constant $c$, which does not depend on $\veps$, since we can use (\ref{aw-reg}).
\end{proof}

\section{Existence for the conservation law}
\begin{theorem}\label{existence}  Assume Assumptions \ref{ass1}, (\ref{ode}) and let $u_0 \in L^1(\Gamma_0)$.
Then there exists an  entropy solution of (\ref{pde}).
\end{theorem}

\begin{proof}
Let $u_\varepsilon$ be the solution of (\ref{regpde}) with
$u_\varepsilon(\cdot,0)=u_{0\varepsilon}$ on $\Gamma_0$ and $u_{0\varepsilon} \to u_0$ a.e. on $\Gamma_0$
, $\Phi$ as in Assumption \ref{ass1}, and
$$w_\varepsilon(y,t):=u_\varepsilon(\Phi(y,t),t)$$ for $y\in \Gamma_0$ and $t\in[0,T]$. Then due to the Lemmata \ref{gradient-absch} and
\ref{udot-absch} and the properties of $\Phi$ (see Assumption \ref{ass1}) we obtain

\bea \int_0^T \int_{\Gamma_0}\Big|\frac{\partial}{\partial t} w_\varepsilon\Big| \le const.
\eea

\bea \int_0^T \int_{\Gamma_0}|\nabla_{\Gamma_0} w_\varepsilon| \le const.
\eea

uniformly in $\varepsilon.$ This implies that $w_\varepsilon$ is uniformly bounded in $H^{1,1}(\Gamma_0\times [0,T])$. Then due to the Kondrakov-Theorem (see Aubin \cite{aubin}) we have a convergent subsequence $w_{\varepsilon'}$ and $w\in L^1(\Gamma_0\times [0,T])$ such that $$w_{\varepsilon'}\to w
\quad \mbox{in}\quad L^{1}(\Gamma_0\times [0,T]).$$ Define $\Phi_t(x):=\Phi(x,t)$.
Since $u_\varepsilon(x,t)=w_\varepsilon(\Phi_t^{-1}(x),t)$ and due to the properties of $\Phi$ (see Assumption \ref{ass1}) we have
$$u_{\varepsilon'}(x,t)\to u(x,t):= w(\Phi_t^{-1}(x),t) \quad \mbox{in} \quad L^1(G_T).$$ Then we proceed as in the proof for Lemma \ref{ent-con} to prove that $u$ is an entropy solution of (\ref{pde}).
\end{proof}
\section{Uniqueness for the conservation law}
In this section we are going to prove uniqueness of entropy solutions (see Theorem \ref{uniqueness})
as defined in Definition \ref{entropy12} or \ref{def:kruzkoventr}.
For the sake of brevity we suppress the integration elements $dx$, $dy$, $dt$ and $d\tau$ the in this section.
Integration is meant to be done over each of these variables that occur in the respective integral.

\newcommand{\abs}[1]{\left\vert#1\right\vert}
\newcommand{\scp}[2]{{ #1 \cdot #2 }}
\newcommand{\real}{\mathbb{R}} % real numbers
\newcommand{\F}{\mathbf{F}}
\renewcommand{\P}{\mathcal P}
\newcommand{\sign}{\operatorname{sign}}
\newcommand{\D}{\operatorname{D}}
\newcommand{\supp}{\operatorname{supp}\hspace{0.05cm}}
\newcommand{\todo}[1]{\textcolor{red}{TODO: #1 !!!}}
\newcommand{\avdots}{\big (\ \vdots \ \big)}

We will need that the initial data is approached in the following sense. The analogous result for the Euclidean case
can be found in \cite{Daf00}.

\begin{lemma}\label{lem:initialdata}
 There is a set $\mathcal E\in[0,T]$ of measure zero such that for $t\in[0,T]\backslash \mathcal E$
an entropy solution $u=u(x,t)$ fulfills
%is defined almost everywhere on $\Gamma(t)$ and
\begin{equation}
 \lim_{\substack{t\rightarrow 0, \\ t\in[0,T]\backslash \mathcal E}} \int_{\Gamma_0} | u(\Phi(x,t),t) - u_0(x)|^2 = 0.
\end{equation}
\end{lemma}

\begin{proof}
We use that an entropy solution is a weak solution and choose $\phi(x,t) := \theta(t)\chi((\Phi(\cdot,t))^{-1}(x))$ in (\ref{eq:weaksol}),
where $\theta(t)= 1-t/\epsilon$ for $0\leq t \leq \epsilon$ and $\theta(t)=0$ for any other $t$ and where
$\chi\in C^\infty(\Gamma_0)$ is arbitrary. Since $C^\infty(\Gamma_0)$ is dense in $L^1(\Gamma_0)$ we get
\begin{equation}\label{eq:weakstarconv}
\int_{\Gamma_0} u_0(x)\chi(x)
= \lim_{\epsilon\rightarrow 0} \frac{1}{\epsilon} \int_0^\epsilon \int_{\Gamma(t)}  u(x,t) \chi((\Phi(\cdot,t))^{-1}(x))
\end{equation}
for all $\chi\in L^1(\Gamma_0)$.
Now we choose $\phi(x,t):= \theta(t)$ as a test function in  \eqref{entropy11} to get
\begin{equation}
- \frac{1}{\epsilon} \int_0^\epsilon \int_{\Gamma(t)} \eta(u(x,t)) + \int_{\Gamma_0} \eta(u_0(x)) \geq \mathcal O(\epsilon).\nonumber
\end{equation}
From \eqref{eq:weakstarconv} we know that
\begin{equation}
\lim_{\epsilon\rightarrow 0} \frac{1}{\epsilon} \int_0^\epsilon \int_{\Gamma(0)}
	  \eta^\prime( u_0(x))(u(\Phi(x,t),t) |\det D\Phi(x,t)| - u_0(x)     ) =0 \nonumber
\end{equation}
and thus
\begin{equation}
\aligned
\lim_{\epsilon\rightarrow 0} \frac{1}{\epsilon} \int_0^\epsilon \int_{\Gamma_0}& \eta(u(\Phi(x,t),t)) |\det D\Phi(x,t)| -  \eta(u_0(x))  \\
	  &- \eta^\prime( u_0(x)) \left(u(\Phi(x,t),t) |\det D\Phi(x,t)| - u_0(x)     \right) \leq 0.
\endaligned \nonumber
\end{equation}
Choosing $\eta(u)=u^2$ and using the fact that $ |\det D\Phi(x,t)| = 1 + \mathcal O (t)$ we arrive at
\begin{equation}
\aligned
\lim_{\epsilon\rightarrow 0} \frac{1}{\epsilon} \int_0^\epsilon \int_{\Gamma_0}
 |u(\Phi(x,t),t)) -  u_0(x) |^2 =0
\endaligned \nonumber
\end{equation}
which implies our claim.
%\todo{checken, dass auch tatsaechlich so ist.}
\end{proof}
For the proof of uniqueness we need  some technical definitions and basic facts from differential geometry.
% In the following we sum over indices occurring twice.
%\begin{lemma}
%\label{lem:localfunctions}
For a parametrization $\tilde{\psi}: U\rightarrow \tilde{\psi}(U)$ of a subset $\tilde{\psi}(U)\subset\Gamma_0$ with $U\subset \real^n$ open
 we have the following properties.
\begin{enumerate}[a)]
 \item For $t>0$ a parametrization of $\Phi(\tilde{\psi}(U),t)\subset \Gamma(t)$ is given by the map $\psi(\cdot,t):U\rightarrow {\psi}(U)$ where
	  \begin{equation}\label{lem:localfunctions1}
           \psi( x, t) := \Phi(\tilde{\psi}(x),t).
          \end{equation}
          \label{lem:localfunctionsa}
 \item For a function $\varphi\in C^1(G_T)$ the material derivative has the local form
      \begin{equation}
        \dot{\varphi}(\psi(x,t),t) = (\varphi(\psi(x,t),t))_t.
       \end{equation}
       \label{lem:localfunctionsb}
 \item For the local representation $W = W_i \; (\psi_{x_i}(\cdot,t)\circ\psi(\cdot,t)^{-1})$ of a tangential vector field $W$ on $\Gamma(t)$ we have
  \begin{equation}
         (\scp{W}{\nabla_{\Gamma(t)}\varphi}) \circ \psi(\cdot,t) = (W_i\circ\psi(t,\cdot)) (\varphi\circ\psi(\cdot,t))_{x_i}
        \end{equation}
      for all $\varphi \in C^1(\Gamma(t))$.
      \label{lem:localfunctionsc}
 \item For any function $\alpha\in L^1(\Gamma(t))$ we have the following local computation of the integral over a subset
	$\psi(V,t)\subset\Gamma(t)$ where $V\subset U$.
      \begin{equation}
        \int_{\psi(V,t)} \alpha = \int_V \alpha \circ \psi(\cdot,t) \sqrt{ \det( \D\psi(\cdot,t)^T \D\psi(\cdot,t) )}.
       \end{equation}
Here $\D$ denotes the Jacobian operator w.r.t. the spatial coordinate $x\in U\subset \real^n$.
      \label{lem:localfunctionsd}
\end{enumerate}
%\end{lemma}
We introduce a function $\delta\in C^\infty(\real)$ satisfying $\delta\geq 0$, $\supp \delta \subset [-1,1]$  and $\int_\real \delta(\sigma)d\sigma = 1$.
For $h>0$ we set
$\delta_h(\sigma):=\frac{1}{h}\delta(\frac{\sigma}{h})$, { then } $\delta_h\in C^\infty(\real)$,  $\supp \delta_h \subset [-h,h]$,
$\delta_h^\prime \leq \frac{C}{h}$ { and } $\int_\real \delta_h(\sigma)d\sigma = 1$ for some constant $C$.
We need the following two Lemmata whose proofs can be found in \cite{Kru70}.
\begin{lemma}\label{lem:kruzkovlip}
If a function $F(u)$ satisfies a Lipschitz condition on an interval $[-M,M]$ with constant $L$, then the function
$H(u,\bar{u}):=\sign(u-\bar{u})[F(u)-F(\bar{u})]$ also satisfies the Lipschitz condition in $u$ and $\bar{u}$ with the constant $L$.
\end{lemma}

\begin{lemma}\label{lem:kruzkovlebesgue}
Let the function $\bar{u}(x,t)$ be bounded and measurable in some cylinder $B_r(0)\times [0,T]$.
If for some $\rho\in \min(0,(r,T))$ and any number $h\in(0,\rho)$ we set
\begin{equation}
 V_h := \frac{1}{h^{n+1}} \iiiint_{\substack{ \left| \frac{t-\tau}{2}\right|\leq h,\; \rho\leq  \frac{t+\tau}{2}\leq T-\rho, \\
					      \left\| \frac{x-y}{2}\right\|\leq h,\; \left\| \frac{x+y}{2}\right\| \leq r-\rho}}
		|\bar{u}(x,t)-\bar{u}(y,\tau)| dxdtdyd\tau,
\end{equation}
then $\lim_{h\rightarrow 0} V_h = 0$.
\end{lemma}

\begin{theorem}\label{uniqueness} Assume Assumptions \ref{ass1} and $\nabla_{\Gamma(t)}\cdot v \in L^\infty(G_T)$ and
 $\dot{\f}$ and $\nabla_{\Gamma(t)}\f$ are Lipschitz continuous with respect to $u$.
 Then the  Kruzkov entropy solution of Definition \ref{def:kruzkoventr} is unique.
\end{theorem}

\begin{proof}
The proof we give here is in the same spirit as Kruzkov's uniqueness proof \cite{Kru70} and can be seen as its extension to moving surfaces.
Let $u,\bar{u}$ be two Kruzkov entropy solutions with initial data $u_0,\bar{u}_0$.
Furthermore let $\psi$ be as in  
\eqref{lem:localfunctions1}
with $\psi(\cdot,t):B_R(0)\rightarrow \psi(B_R(0),t)$
where $B_R(0)\subset\real^n$ and  $\psi(B_R(0),t)\subset \Gamma(t)$. Choose normal coordinates for instance.
For some $0<r<R$ we define $\Omega:= B_r(0)$ and $\Omega_T:= B_r(0) \times [0,T]$.
Let now $\varphi\in C^\infty(G_T)$ be a test function with $\varphi\geq 0$ and
$\supp \varphi \Subset \bigcup_{t\in (\rho,T-2\rho)} \psi(\Omega,t)\times \{t\}$ where $\rho>0$ is a small
real number. We know because of \ref{lem:localfunctionsa})\ --\ \ref{lem:localfunctionsd}) that then
\begin{equation}\label{eq:kruzkovlocal}
%\begin{split}
\aligned
 \int_{G_T} & \abs{u(x,t)-k} \dot{\varphi} - \sign(u(x,t) - k ) \;k\; \nabla_{\Gamma(t)} \cdot v(x,t) \;\varphi (x,t)\\
&+ \sign(u(x,t)-k) \scp{ (\f((x,t),u(x,t))-\f((x,t),k))} { \nabla_{\Gamma(t)} \varphi(x,t) }  \\
&+ \int_{\Gamma_0} \abs{u_0(x)-k} \varphi(x,0) \\
=  \int_0^T \int_{\Omega} & [\abs{u(\psi(x,t),t)-k} (\varphi(\psi(x,t),t)_t\\
 &- \sign(u(\psi(x,t),t) - k ) \;k\;  (\nabla_{\Gamma(t)} \cdot v)(\psi(x,t),t)  \;\varphi (\psi(x,t),t)\\
&+ \sign(u(\psi(x,t),t)-k)\\
&\quad \cdot (\f_i((\psi(x,t),t),u(\psi(x,t),t))-\f_i((\psi(x,t),t),k)) ( \varphi(\psi(x,t),t))_{x_i}]
 \sqrt{ \det( g(x,t) )}
\endaligned
%\end{split}
\end{equation}
where $\f_i$ is locally defined by $\f = \f_i \cdot (\psi_{x_i}(\cdot,t)\circ\psi(\cdot,t)^{-1})$ and
$g(x,t) := \D\psi(x,t)^T \D\psi(x,t)$, where $\D\psi(x,t)$ denotes the Jacobian of $\psi(x,t)$ with respect to $x$.
For better readability we will suppress the composition with
$\psi(\cdot,t)$ in the following, i.e. we introduce new functions which live on $\Omega\times[0,T]$
and which we mainly again denote by the names of the original functions. By this we mean to do the following replacements.
$u(\psi(x,t),t) \rightsquigarrow u(x,t)$, $\varphi(\psi(x,t),t) \rightsquigarrow \varphi(x,t)$,
$ (\nabla_{\Gamma(t)} \cdot w)(\psi(x,t),t)\rightsquigarrow q(x,t)$
and $\f_i((\psi(x,t),t),\cdot)\rightsquigarrow \f_i(x,t,\cdot)$.

In Definition \ref{def:kruzkoventr} we choose a test function
$$\tilde \varphi = \tilde \varphi(x,t,y,\tau) \geq 0 \quad \mbox{with} \quad
\supp\tilde \varphi \Subset (\bigcup_{t\in (\rho,T-2\rho)} \psi(\Omega,t)\times \{t\})^2,$$ set $k=\bar{u}(y,\tau)$,
multiply with $\sqrt{ \det( g(y,\tau) )}$ and integrate over $\Omega_T:=\Omega\times[0,T]$ with respect to $(y,\tau)$. Using \eqref{eq:kruzkovlocal}
we arrive at
\begin{equation}\label{eq:kruzkovdoubleu}
%\begin{split}
\aligned
\int_{\Omega_T^2} & [\abs{u(x,t)-\bar{u}(y,\tau)} \tilde\varphi_t(x,t,y,\tau)\\
 &- \sign(u(x,t) - \bar{u}(y,\tau) ) \;\bar{u}(y,\tau)\;  q(x,t)  \;\tilde\varphi (x,t,y,\tau)\\
&+ \sign(u(x,t)-\bar{u}(y,\tau)) (\f_i(x,t,u(x,t))-\f_i(x,t,\bar{u}(y,\tau))) \tilde\varphi_{x_i}(x,t,y,\tau)]\\
& \sqrt{ \det( g(x,t) )} \sqrt{ \det( g(y,\tau) )} \geq 0.
\endaligned
%\end{split}
\end{equation}
Proceeding analogously for the corresponding version of \eqref{eq:krukoventr} for the entropy solution $\bar{u}=\bar{u}(y,\tau)$ we get
\begin{equation}\label{eq:kruzkovdoublev}
%\begin{split}
\aligned
\int_{\Omega_T^2} & [\abs{u(x,t)-\bar{u}(y,\tau)} \tilde\varphi_\tau(x,t,y,\tau)\\
 &+ \sign(u(x,t) - \bar{u}(y,\tau) ) \;u(x,t)\;  q(y,\tau)  \;\tilde\varphi (x,t,y,\tau)\\
&+ \sign(u(x,t)-\bar{u}(y,\tau)) (\f_i(y,\tau,u(x,t))-\f_i(y,\tau,\bar{u}(y,\tau)))  \tilde\varphi_{y_i}(x,t,y,\tau)]\\
& \sqrt{ \det( g(x,t) )} \sqrt{ \det( g(y,\tau) )} \geq 0.
\endaligned
%\end{split}
\end{equation}
Summing up \eqref{eq:kruzkovdoubleu} and \eqref{eq:kruzkovdoublev} one sees
\begin{equation}\label{eq:kruzkovdoubleuv}
%\begin{split}
\aligned
\int_{\Omega_T^2} & [\abs{u(x,t)-\bar{u}(y,\tau)} (\tilde\varphi_t(x,t,y,\tau) + \tilde\varphi_\tau(x,t,y,\tau))\\
&+ \sign(u(x,t)-\bar{u}(y,\tau)) (\f_i(x,t,u(x,t))-\f_i(y,\tau,\bar{u}(y,\tau))) (  \tilde\varphi_{x_i}(x,t,y,\tau) +  \tilde\varphi_{y_i}(x,t,y,\tau))\\
&+ \underbrace{\sign(u(x,t)-\bar{u}(y,\tau)) (\f_i(y,\tau,\bar{u}(y,\tau))-\f_i(x,t,\bar{u}(y,\tau))) \tilde\varphi_{x_i}(x,t,y,\tau)}_{=:R_{1}} \\
&+ \underbrace{\sign(u(x,t)-\bar{u}(y,\tau)) (\f_i(y,\tau,u(x,t))-\f_i(x,t,u(x,t))) \tilde\varphi_{y_i}(x,t,y,\tau)}_{=:R_{2}} \\
&+ \underbrace{\sign(u(x,t) - \bar{u}(y,\tau) ) (u(x,t)  q(y,\tau) - \bar{u}(y,\tau) q(x,t))}_{=:Q}  \tilde\varphi (x,t,y,\tau)]\\
& \sqrt{ \det( g(x,t) )} \sqrt{ \det( g(y,\tau) )} \geq 0.
\endaligned
%\end{split}
\end{equation}
For a test function $\varphi=\varphi(x,t)\geq 0$ satisfying $\supp\varphi \Subset \bigcup_{t\in (\rho,T-2\rho)} \psi(\Omega,t)\times \{t\}$
we set (in local coordinates)
\begin{equation}\label{eq:defphitilde}
\tilde \varphi(x,t,y,\tau) := \varphi\left(\frac{x+y}{2},\frac{t+\tau}{2}\right) \delta_h\left( \frac{t-\tau}{2} \right )
  \prod_{i=1}^n \delta_h\left (\frac{x_i-y_i}{2}   \right) \ =: \ \varphi(\cdots) \lambda_h \avdots
\end{equation}
where
\begin{equation}
(\cdots) =  \left(\frac{x+y}{2},\frac{t+\tau}{2}\right), \qquad \avdots = \left(\frac{x-y}{2},\frac{t-\tau}{2}\right)
\end{equation}
and $h$ is sufficiently small, such that $\supp\tilde \varphi \Subset (\bigcup_{t\in (\rho,T-2\rho)} \psi(\Omega,t)\times \{t\})^2 $.
For the partial derivatives of $\tilde \varphi$ the following identities are trivial.
\begin{equation}
 \aligned
\tilde\varphi_t + \tilde\varphi_\tau &= \varphi_t(\cdots) \lambda_h\avdots, \\
\tilde\varphi_{x_i} + \tilde\varphi_{y_i} &= \varphi_{x_i}(\cdots) \lambda_h \avdots.
\endaligned
\end{equation}
The major part of the proof will be to see that with $\tilde\varphi$ as in \eqref{eq:defphitilde} the following inequality is
obtained from \eqref{eq:kruzkovdoubleuv} for $h\rightarrow 0$:
\begin{equation}\label{eq:kruzkovcontracted}
%\begin{split}
\aligned
\int_{\Omega_T} & [\abs{u(x,t)-\bar{u}(x,t)} \varphi_t(x,t) \\
&+ \sign(u(x,t)-\bar{u}(x,t)) (\f_i(x,t,u(x,t))-\f_i(x,t,\bar{u}(x,t))) \varphi_{x_i}(x,t)\\
&- \sign(u(x,t)-\bar{u}(x,t)) ({\f_i}_{x_i}(x,t,u(x,t))-{\f_i}_{x_i}(x,t,\bar{u}(x,t))) \varphi(x,t)\\
&+|u(x,t) - \bar{u}(x,t) | q(x,t) \varphi (x,t)] \det( g(x,t) )  \geq 0.\\
\endaligned
%\end{split}
\end{equation}
In order to prove \eqref{eq:kruzkovcontracted} we define a function
\begin{equation}
 \aligned
P_h(x,t,y,\tau) :=& F(x,t,y,\tau,u(x,t),\bar{u}(y,\tau)) \lambda_h\avdots\\
:=&[\abs{u(x,t)-\bar{u}(y,\tau)} \varphi_t(\cdots) \\
&+ \sign(u(x,t)-\bar{u}(y,\tau)) (\f_i(x,t,u(x,t))-\f_i(y,\tau,\bar{u}(y,\tau))) \varphi_{x_i}(\cdots) \\
&+ Q \varphi (\cdots)] \sqrt{ \det( g(x,t) )} \sqrt{ \det( g(y,\tau) )} \; \lambda_h\avdots.
\endaligned
\end{equation}
From \eqref{eq:kruzkovdoubleuv} we deduce
\begin{equation}
 \aligned
0 \leq& \int_{\Omega_T^2} F(x,t,y,\tau,u(x,t),\bar{u}(y,\tau)) \lambda_h\avdots \\
     &+\int_{\Omega_T^2} [R_1 + R_2]  \sqrt{ \det( g(x,t) )} \sqrt{ \det( g(y,\tau) )}\\
     =&\int_{\Omega_T^2} [F(x,t,y,\tau,u(x,t),\bar{u}(y,\tau)) - F(x,t,x,t,u(x,t),\bar{u}(x,t))] \lambda_h\avdots \\
     &+\int_{\Omega_T^2}  F(x,t,x,t,u(x,t),\bar{u}(x,t)) \lambda_h\avdots \\
     &+\int_{\Omega_T^2} [R_1 + R_2]  \sqrt{ \det( g(x,t) )} \sqrt{ \det( g(y,\tau) )} = : J_1(h) + J_2 + J_3(h).
\endaligned
\end{equation}
We mention that due to Lemma \ref{lem:kruzkovlip} and since $\bar\Omega_T\subset B_R(0)\times [0,T]$ and $F$ is defined on $B_R(0)\times [0,T]$,
obviously $F$ is Lipschitz continuous on $\Omega_T$ in all its arguments. We then use the fact that $|\lambda_h\avdots|\leq C h^{-(n+1)}$
and Lemma \ref{lem:kruzkovlebesgue} to see that
\begin{equation}
 \aligned
|J_1(h)| &\leq C \left(h + \int_{\Omega_T^2}   |\bar{u}(y,\tau) - \bar{u}(x,t)|\lambda_h\avdots\right) \rightarrow 0 \\
\endaligned
\end{equation}
for $h\rightarrow 0$ since $\bar{u}$ is bounded and measurable.
By substitution we get for the second term
\begin{equation}
\aligned
 J_2 &= \int_{\Omega_T} F(x,t,x,t,u(x,t),\bar{u}(x,t)) \underbrace{\int_{\Omega_T}   \lambda_h\left(\frac{x-y}{2},\frac{t-\tau}{2}\right) \;
%\ d\tau}_{=2^{n+1}} \ dx \ dt\\
\ }_{=2^{n+1}} \ \\
     &= 2^{n+1} \int_{\Omega_T} F(x,t,x,t,u(x,t),\bar{u}(x,t))   \ dx \ dt.
     \endaligned
\end{equation}
Now we turn to the third term
\begin{equation}
 \aligned
J_3(h) =& \int_{\Omega_T^2} \Big[  \sign(u(x,t)-\bar{u}(y,\tau)) (\f_i(y,\tau,\bar{u}(y,\tau))-\f_i(x,t,\bar{u}(y,\tau)))\\
	&\qquad\qquad\cdot\left(\frac{1}{2}\varphi_{x_i}(\cdots) \lambda_h \avdots +  \varphi(\cdots) \left(\lambda_h \avdots\right)_{x_i}\right)\\
&+ \sign(u(x,t)-\bar{u}(y,\tau)) (\f_i(y,\tau,u(x,t))-\f_i(x,t,u(x,t)))  \\
	&\qquad\qquad\cdot\left(\frac{1}{2}\varphi_{y_i}(\cdots) \lambda_h \avdots +  \varphi(\cdots) \left(\lambda_h \avdots\right)_{y_i}\right)\Big]\\
& \sqrt{ \det( g(x,t) )} \sqrt{ \det( g(y,\tau) )}.
\endaligned
\end{equation}
Here, we notice that those summands of the above integral that contain $\varphi_{x_i}(\cdots) \lambda_h \avdots$ or $\varphi_{y_i}(\cdots) \lambda_h \avdots$
as a factor in the integrand vanish for $h\rightarrow 0$. Therefore, it suffices to show that
\begin{equation}
 \aligned
\hat J_3(h) :=& \int_{\Omega_T^2} \sign(u(x,t)-\bar{u}(y,\tau)) \Big[   (\f_i(y,\tau,\bar{u}(y,\tau))-\f_i(x,t,\bar{u}(y,\tau))) \left(\lambda_h \avdots\right)_{x_i}\\\\
&+ (\f_i(y,\tau,u(x,t))-\f_i(x,t,u(x,t)))
	 \left(\lambda_h \avdots\right)_{y_i}\Big]\hat\varphi(x,t,y,\tau)\\
\endaligned
\end{equation}
converges to
\begin{equation}
-2^{n+1} \int_{\Omega_T} \sign(u(x,t)-\bar{u}(x,t)) ({\f_i}_{x_i}(x,t,u(x,t))-{\f_i}_{x_i}(x,t,\bar{u}(x,t))) \varphi(x,t)\det( g(x,t) )\\
\end{equation}
for $h\rightarrow 0$, where $\hat\varphi(x,t,y,\tau) :=\varphi(\cdots) \sqrt{ \det( g(x,t) )} \sqrt{ \det( g(y,\tau) )}$.
For a better readability we write $\lambda_{x_i}$ instead of $\Big(\lambda_h \avdots\Big)_{x_i}$ and analogously $\lambda_{y_i}$.
Due to the local Lipschitz continuity of $\f_i$ on $\Omega_T$ we have
\begin{equation}
 \aligned
\left(\f_i(y,\tau,\bar{u}(y,\tau)) - \f_i(x,t,\bar{u}(y,\tau)) \right) \lambda_{x_i}
= {\f_i}_{\tau}(y,\tau,\bar{u}(y,\tau)) (\tau-t) \lambda_{x_i} \\+ {\f_i}_{y_j}(y,\tau,\bar{u}(y,\tau)) (y_j-x_j) \lambda_{x_i} + \epsilon_i \lambda_{x_i}
\endaligned
\end{equation}
with $\frac{\epsilon_i}{d}\rightarrow 0$ for $d:=\|x-y\|+|t-\tau| \rightarrow 0$ where $\|\cdot \|$ denotes the Euclidean norm in $\real^n$.
Analogously, with $\lambda_{x_i} = - \lambda_{y_i}$ we get
\begin{equation}
 \aligned
\left(\f_i(y,\tau,u(x,t)) - \f_i(x,t,u(x,t)) \right) \lambda_{y_i}
= - {\f_i}_{\tau}(y,\tau,u(x,t)) (\tau-t) \lambda_{x_i} \\- {\f_i}_{y_j}(y,\tau,u(x,t)) (y_j-x_j) \lambda_{x_i} + \beta_i \lambda_{x_i}
\endaligned
\end{equation}
with $\frac{\beta_i}{d}\rightarrow 0$ for $d \rightarrow 0$.
Thus,
\begin{equation}
 \aligned
\hat J_3(h)& = \int_{\Omega_T^2} \hat\varphi(x,t,y,\tau)\sign(u(x,t)-\bar{u}(y,\tau))
		     ({\f_i}_\tau(y,\tau,\bar{u}(y,\tau))-{\f_i}_\tau(y,\tau,u(x,t))) (\tau - t ) \lambda_{x_i} \\
&+ \int_{\Omega_T^2} \hat\varphi(x,t,y,\tau)\sign(u(x,t)-\bar{u}(y,\tau))
		     ({\f_i}_{y_j}(y,\tau,\bar{u}(y,\tau))-{\f_i}_{y_j}(y,\tau,u(x,t))) (y_j - x_j ) \lambda_{x_i} \\
&+ \int_{\Omega_T^2} \hat\varphi(x,t,y,\tau)\sign(u(x,t)-\bar{u}(y,\tau))   (\epsilon_i + \beta_i ) \lambda_{x_i} =: I_1 + I_2 + I_3\\
\endaligned
\end{equation}
Obviously, $|\lambda_{x_i}| \leq C h^{-(n+2)}$ and using Lemma \ref{lem:kruzkovlebesgue} we get
$|I_3| \rightarrow 0$ as $h\rightarrow 0$.
Since $\sqrt{\det(g)}$ is Lipschitz continuous on $\Omega_T$ we have
\begin{equation}
 |\hat \varphi (x,t,y,\tau) - \hat \varphi (x,t,x,t) | \leq C (\|x-y\| + |t-\tau|)
\end{equation}
and obtain
\begin{equation}
 \aligned
I_1 = \int_{\Omega_T^2}[ \hat\varphi(y,\tau,y,\tau)\sign(u(x,t)-\bar{u}(y,\tau))
		     ({\f_i}_\tau(y,\tau,\bar{u}(y,\tau))-{\f_i}_\tau(y,\tau,u(x,t))) (\tau - t ) \lambda_{x_i}] \; + \gamma_1(h)\\
\endaligned
\end{equation}
with $\gamma_1(h)\rightarrow 0$ for $h\rightarrow 0$. Defining
\begin{equation}
 F_i(y,\tau,u(x,t),\bar{u}(y,\tau)) := \hat\varphi(y,\tau,y,\tau)\sign(u(x,t)-\bar{u}(y,\tau)) ({\f_i}_\tau(y,\tau,\bar{u}(y,\tau))-{\f_i}_\tau(y,\tau,u(x,t)))
\end{equation}
we see with Lemma \ref{lem:kruzkovlip} that $F_i$ is Lipschitz continuous in $u$ on $\Omega_T$ and obtain according to Lemma \ref{lem:kruzkovlebesgue}
\begin{equation}
 \aligned
|I_1|& \leq\left| \int_{\Omega_T^2}  F_i(y,\tau,u(x,t),\bar{u}(y,\tau)) ((\tau - t ) \lambda)_{x_i}\right|  \; +\; |\gamma_1(h)| \\
&= \left|\int_{\Omega_T^2}  (F_i(y,\tau,u(x,t),\bar{u}(y,\tau)) - F_i(y,\tau,u(y,\tau),\bar{u}(y,\tau))) ((\tau - t ) \lambda)_{x_i} \right| \; + \;|\gamma_1(h)| \\
&\leq C \int_{\Omega_T^2} |u(x,t) - u(y,\tau) | \underbrace{|(\tau - t )|}_{\leq h} \underbrace{|\lambda_{x_i}|}_{\leq Ch^{-(n+2)}} \; + \;|\gamma_1(h)| \; \rightarrow 0
\endaligned
\end{equation}
for $h\rightarrow 0$. Considering $I_2$ we have
\begin{equation}
 \aligned
I_2& =\int_{\Omega_T^2} \hat\varphi(x,t,y,\tau)\sign(u(x,t)-\bar{u}(y,\tau))
		     ({\f_i}_{y_j}(y,\tau,\bar{u}(y,\tau))-{\f_i}_{y_j}(y,\tau,u(x,t))) ((y_j - x_j ) \lambda)_{x_i} \\
&+\int_{\Omega_T^2} \hat\varphi(x,t,y,\tau)\sign(u(x,t)-\bar{u}(y,\tau))
		     ({\f_i}_{y_i}(y,\tau,\bar{u}(y,\tau))-{\f_i}_{y_i}(y,\tau,u(x,t)))  \lambda) =: I_{2.1} + I_{2.2}\\
\endaligned
\end{equation}
We see that $I_{2.2}$ converges for $h\rightarrow 0 $ to
\begin{equation}
 2^{n+1}\int_{\Omega_T} \sign(u(x,t)-\bar{u}(x,t)) ({\f_i}_{x_i}(x,t,\bar{u}(x,t))-{\f_i}_{x_i}(x,t,u(x,t))) \varphi(x,t)\det( g(x,t) ),
\end{equation}
whereas one can see analogously to $I_1$ that $I_{2.1}$ converges to zero.
Thus, we conclude that
\begin{equation}
\aligned
 0\leq& \lim_{h\rightarrow 0}\left[ J_1(h) + J_2 + J_3(h)\right] = 2^{n+1}\int_{\Omega_T} F(x,t,x,t,u(x,t),\bar{u}(x,t))\\
&-2^{n+1}\int_{\Omega_T} \sign(u(x,t)-\bar{u}(x,t)) ({\f_i}_{x_i}(x,t,u(x,t))-{\f_i}_{x_i}(x,t,\bar{u}(x,t))) \varphi(x,t)\det( g(x,t) )
\endaligned
\end{equation}
and thereby \eqref{eq:kruzkovcontracted}.
In order to continue the proof of the theorem we introduce the following definitions.
Let $\mathcal E_u\subset[0,T]$ be of
$\mathcal L^1$-measure zero, such that $u(\Phi(\cdot,t),t)\rightarrow u_0$ in $L^1(\Gamma_0)$ for $t\rightarrow 0$, $t\in [0,T]\backslash\mathcal E_u$.
Let $\mathcal E_v$ be defined analogously. These sets exist because of Lemma \ref{lem:initialdata}. We set
\begin{equation}
 \mu(t) := \int_{S_t} | u(x,t) - \bar{u}(x,t) |\; \det(g(x,t)),
\end{equation}
where $S_t:= \left\{ x\in B_r(0) \big| ||x|| \leq r-Lt\right\} $ and
\begin{equation}\label{eq:lipschitzf}
 L:= \max_{\substack{(x,t)\in \bar \Omega_T,\\ w \leq \max(\|u\|_{L^\infty},\|\bar{u}\|_{L^\infty}) }}
		  \left(  \sum_{i=1}^n \partial_u\f_i^2(x,t,w)\right)^{\frac{1}{2}}
\end{equation}
By $\mathcal E_\mu\subset [0,T]$ we denote those points that are not Lebesgue points of the bounded and measurable function $\mu$ and set
$\mathcal E_0:=\mathcal E_u \cup \mathcal E_{\bar{u}} \cup \mathcal E_\mu$. Obviously, $\mathcal L^1 (\mathcal E_0) = 0$. We define
\begin{equation}
 a_h(\sigma) := \int_{-\infty}^\sigma \delta_h(\xi)d\xi
\end{equation}
as a regularization of the Heavyside function and see $a_h^\prime(\sigma) = \delta_h(\sigma)\geq 0$.
Let now $\tau_1,\tau_2\in(0,T)\backslash \mathcal E_0$ with $\tau_1<\tau_2$.
In order to prove a contraction property we choose the following test function whose definition is given in local coordinates by
\begin{equation}\label{eq:defphi}
\varphi(x,t):= [a_h(t-\tau_1) - a_h(t-\tau_2)]\zeta(x,t),
\end{equation}
where
\begin{equation}
 \zeta(x,t) := \zeta_\epsilon(x,t) := 1 - a_\epsilon( \|x\| + Lt - r + \epsilon ) , \qquad \epsilon>0
\end{equation}
for $(x,t) \in \Omega_T$ and zero outside.
Hence, $\supp\zeta(\cdot,t) \subset S_t \subset B_r(0)$.
We compute the derivatives of $\zeta$ as
\begin{equation}
 \zeta_t(x,t) :=-L \underbrace{a^\prime_\epsilon( \|x\| + Lt - r + \epsilon ) }_{\geq 0} \text{ and }
\nabla\zeta(x,t) := - a^\prime_\epsilon( \|x\| + Lt - r + \epsilon ) \frac{x}{\|x\|}
\end{equation}
and due to the definition of $L$ we conclude with the Cauchy Schwartz inequality
\begin{equation}
 0=\zeta_t(x,t) + L \| \nabla\zeta(x,t) \| \geq \zeta_t(x,t) + \frac{\f_i(x,t,u(x,t)) - \f_i(x,t,\bar{u}(x,t))}{u(x,t) - \bar{u}(x,t)} \zeta_{x_i}(x,t).
\end{equation}
If we choose the function $\varphi$ from \eqref{eq:defphi} as a test function in \eqref{eq:kruzkovcontracted} we get
\begin{equation}
\aligned
 &\int_{\Omega_T}[\delta_h(t-\tau_1) - \delta_h(t-\tau_2)]\zeta(x,t) |u(x,t) - \bar{u}(x,t)| \det(g(x,t)) \\
+ &\int_{\Omega_T}\underbrace{[a_h(t-\tau_1) - a_h(t-\tau_2)]}_{\geq 0} \underbrace{|u(x,t) - \bar{u}(x,t)|}_{\geq0}\\
    &\qquad \underbrace{\left[\zeta_t(x,t) + \frac{\f_i(x,t,u(x,t)) - \f_i(x,t,\bar{u}(x,t))}{u(x,t) - \bar{u}(x,t)} \zeta_{x_i}(x,t) \right]}_{\leq 0}
		\underbrace{\det(g(x,t))}_{> 0} \\
&-\int_{\Omega_T} \sign(u(x,t)-\bar{u}(x,t)) ({\f_i}_{x_i}(x,t,u(x,t))-{\f_i}_{x_i}(x,t,\bar{u}(x,t))) \\
     &\qquad \cdot [a_h(t-\tau_1) - a_h(t-\tau_2)]\zeta(x,t) \det( g(x,t) )\\
+&\int_{\Omega_T} |u(x,t) - \bar{u}(x,t) | q(x,t) [a_h(t-\tau_1) - a_h(t-\tau_2)]\zeta(x,t) \det( g(x,t) )  \geq 0.\\
\endaligned
\end{equation}
For $\epsilon\rightarrow 0 $ we have
\begin{equation}
\aligned
 \int_0^T& [\delta_h(t-\tau_1) - \delta_h(t-\tau_2)]  \int_{S_t} |u(x,t) - \bar{u}(x,t)| \det(g(x,t)) \\
-\int_0^T &[a_h(t-\tau_1) - a_h(t-\tau_2)]  \int_{S_t} \sign(u(x,t)-\bar{u}(x,t))\\
&\qquad \qquad \qquad \cdot ({\f_i}_{x_i}(x,t,u(x,t))-{\f_i}_{x_i}(x,t,\bar{u}(x,t))) \det(g(x,t)) \\
&+\int_0^T  [a_h(t-\tau_1) - a_h(t-\tau_2)] \int_{S_t} |u(x,t) - \bar{u}(x,t) | q(x,t) \det( g(x,t) )  \geq 0.\\
\endaligned
\end{equation}
If now $h\rightarrow 0$ this implies
\begin{align}
 \mu(\tau_2)& = \int_{S_{\tau_2}} |u(x,\tau_2) - \bar{u}(x,\tau_2)| \det(g(x,\tau_2)) \\
&\leq\int_{S_{\tau_1}} |u(x,\tau_1) - \bar{u}(x,\tau_1)| \det(g(x,\tau_1)) \\
&\qquad+  \int_{\tau_1}^{\tau_2}  \int_{S_t} |u(x,t) - \bar{u}(x,t) | (|q(x,t)|+L({\f_i}_{x_i})) \det( g(x,t) )  \\
&\leq\mu(\tau_1)  + (\|q\|_{L^\infty}+L({\f_i}_{x_i})) \int_{\tau_1}^{\tau_2} \mu(t),
\end{align}
where $L({\f_i}_{x_i})$ denotes the local Lipschitz constant of ${\f_i}_{x_i}$ on $\Omega_T$ and by a Gronwall argument we conclude
\begin{equation}\label{eq:localconstraction}
 \mu(\tau_2) \leq \mu(\tau_1) \exp( (\|q\|_{L^\infty}+L({\f_i}_{x_i})) (\tau_2 - \tau_1)).
\end{equation}
Using the fact that
\begin{equation}
 \big| |u-\bar{u}| - |u_0-\bar{u}_0| \big| \leq |u_0-u| + |\bar{u}_0-\bar{u}|
\end{equation}
and that $\det(g(x,t))$ is bounded and Lipschitz in $t$ we get with Lemma  \ref{lem:initialdata}
for $\tau_1\rightarrow 0$ in $ [0,T]\backslash \mathcal E_0$ the following estimate
\begin{equation} \label{eq:localconstractioninitial}
 \mu(\tau_2) \leq  \exp( (\|q\|_{L^\infty}+L({\f_i}_{x_i})) \tau_2) \:  \int_{B_r(0)} | u_0(x) - \bar{u}_0(x) |\; \det(g(x,0))
\end{equation}
for all $\tau_2 \in  [0,T]\backslash \mathcal E_0$.
At this point we are able to show that $u=\bar{u}$ almost everywhere if $u_0=\bar{u}_0$ almost everywhere.
To this end we assume that $u_0=\bar{u}_0$ almost everywhere. Let now $p\in \Gamma_0$. We show that
we find an open set $\tilde U_p\subset \Gamma_0$ containing $p$ such that $u=\bar{u}$ almost everywhere
in $\bigcup_{t\in[0,\tilde t]} \Phi(\tilde U_p,t) \times \{t\}$ for some $\tilde t >0$. To this end let again $\psi$ be as in
\eqref{lem:localfunctions1} with $\psi(\cdot,t):B_R(0)\rightarrow \psi(B_R(0),t)$
where $B_R(0)\subset\real^n$, $\psi(B_R(0),t)\subset \Gamma(t)$ and $\psi(0,0) = p$.
Furthermore we choose $0<r<R$  and set $\tilde \psi := \psi(\cdot,0)$ and
$\tilde U := \tilde \psi(B_r(0))\subset \Gamma_0$.
As in \eqref{eq:lipschitzf} we get a local Lipschitz constant $L$ of $\f$ and have
for $\tilde t:=\frac{r}{2L}$ that
\begin{equation}
 \int_{S_{t}} |u(x,t) - \bar{u}(x,t) | \det( g(x,t) ) \leq
 C \int_{B_{r}(0)} |u_0(x) - \bar{u}_0(x) | \det( g(x,0) ) =0
\end{equation}
for all $t\in[0,\tilde t]$ and therefore $u=\bar{u}$ almost everywhere in $\bigcup_{t\in[0,\tilde t]} \Phi(\tilde U_p,t) \times \{t\}$
with $\tilde U_p:= \tilde \psi (B_{\frac{r}{2}}(0))$ since $S_{\tilde t}=B_{\frac{r}{2}}(0)$. Repeating the above argumentation for every
$p\in \Gamma_0$ we get several sets $\tilde U_p$ whose union obviously covers $\Gamma_0$. By choosing a finite
cover $\left\{ \tilde U_{p_i}, \;i=1,\ldots,M\right\} $ we get several times $\tilde t_i$ such that
$u=\bar{u}$ almost everywhere in $\bigcup_{t\in[0,\tilde t_{\min}]} \Gamma(t) \times \{t\}$ with
$\tilde t_{\min}:=\min_{i\in \{1,\ldots,M\}}\tilde t_i$. As $\tilde t_{\min}$ does not depend on time and
because of \eqref{eq:localconstraction} we can successively
conclude that $u=\bar{u}$ almost everywhere in $G_T$.
\end{proof}

The following Corollary follows immediately from the equivalence of Definition \ref{entropy12} and Definition \ref{def:kruzkoventr} together with Theorem \ref{uniqueness}.
\begin{cor}
 Let the assumptions of Theorem \ref{uniqueness} be satisfied. Then the entropy solution of Definition \ref{entropy12} is unique .
\end{cor}

\section{Numerical algorithm}
Now we are going to derive a finite volume scheme for the initial value problem (\ref{pde}).
Up to our knowledge the first finite volume scheme on evolving surfaces for parabolic equations was proposed by Lenz et al. \cite{lenz}.
They provide a scheme for diffusion on evolving surfaces. We adapt this scheme to nonlinear scalar conservation laws on
evolving surfaces.

\subsection{Notation and Preliminaries}
Following Dziuk and Elliot \cite{dziuk1} the smooth initial surface $\Gamma_0$ is approximated by a triangulated surface
$\Gamma_{0,h}$ which consists of a set of simplices (triangles for $n=2$) such that all its vertices $\{x^0_j\}_{j=1}^N$ sit on $\Gamma_0$.
Such a set of simplices is called a triangulation
$\T^0_h$ of $\Gamma_{0,h}$ and $h$ indicates the maximal diameter of a triangle on the whole family of triangulations.
The triangulation $\T_h(t)$ and its $\Gamma(t)$ approximating surface $\Gamma_h(t)$ is defined by mapping the set
of vertices $\{x^0_j\}_{j=1}^N$ with $\Phi(\cdot,t)$ onto $\Gamma(t)$, i.e.
\begin{align*}
x_j(t):=\Phi(x^0_j,t),
\end{align*}
i.e. they lie on motion trajectories. Thus, all the triangulations $\T_h(t)$ share the same grid topology.\\
By this construction the set of simplices can be written as $\T_h(t)=\{T_j(t)|j=1,\ldots,M\}$ for $t\in[0,T]$,
where $M$ is the time independent number of simplices.

For the derivation of a finite volume scheme we introduce discrete time steps $t^k=k\tau$ where $\tau$ denotes
the time step size and $k$ the time step index. For an arbitrary time step $t^k$ we have a smooth surface
$\Gamma^k:=\Gamma(t^k)$, its approximation $\Gamma_h^k:=\Gamma_h(t^k)$ and the corresponding triangulation $\T^k_h:=\T_h(t^k)$
with simplices $T^k_j:=T_j(t^k)$.

From \cite{dziuk1} we know that for sufficiently small $h$ there is a uniquely defined lifting operator from
$\Gamma^k_h$ onto $\Gamma^k$ via the orthogonal projection $\P^k=\P(t^k)$ in direction of the surface normal $\nu$ on $\Gamma^k$.

For the comparison of quantities on $\Gamma^k$ and on $\Gamma^k_h$ we define curved simplices via the projection operator, i.e.
\begin{align*}
  \mathfrak{T}^k_j := \P^k T^k_j.
\end{align*}
Such a projection $\mathfrak{T}^k_j$ propagates during the $(k+1)$th time interval $[t^k,t^{k+1}]$. We define
\begin{align*}
  \mathfrak{T}^k_j(t) := \Phi(  (\Phi(\cdot,t^k))^{-1}(\mathfrak{T}^k_j)     ,t)
\end{align*}
and denote by $V^k_j$ the $n$-dimensional measure of $T^k_j$.

\subsection{Definition of the Finite Volume Scheme}
In order to derive a suitable finite volume scheme we integrate (\ref{pde}) over the curved simplex $\mathfrak{T}^k_j(t)$
and the time interval $[t^k,t^{k+1}]$. Applying the Leibniz formula and the Gauss theorem we obtain %for the material derivative
\begin{align}
 0&=\int_{t^k}^{t^{k+1}} \int_{\mathfrak{T}^k_j(t)}  \dot{u} + u\nabla_{\Gamma_t}\cdot v + \nabla_{\Gamma_t}\cdot \f(\cdot,\cdot,u) \nonumber\\
 &=\int_{t^k}^{t^{k+1}}  \frac{d}{dt} \int_{\mathfrak{T}^k_j(t)}  u
              + \int_{t^k}^{t^{k+1}}  \int_{\partial \mathfrak{T}^k_j(t)} \f(\cdot,\cdot,u) \cdot \nu_{\partial \mathfrak{T}^k_j(t)} \nonumber\\
&= \int_{\mathfrak{T}^k_j(t^{k+1})}  u - \int_{\mathfrak{T}^k_j(t^{k})}  u
              + \int_{t^k}^{t^{k+1}} \sum_{\mathfrak e \subset \partial \mathfrak{T}^k_j(t)}
                            \int_{\mathfrak e} \f(\cdot,\cdot,u) \cdot \nu_{\partial \mathfrak{T}^k_j(t)}\nonumber\\
&= \int_{\mathfrak{T}^k_j(t^{k+1})}  u - \int_{\mathfrak{T}^k_j}  u
              + \int_{t^k}^{t^{k+1}} \sum_{\mathfrak e \subset \partial \mathfrak{T}^k_j(t)}
                            \int_{\mathfrak e} \f(\cdot,\cdot,u) \cdot \nu_{\partial \mathfrak{T}^k_j(t)},\nonumber
\end{align}
where $\nu_{\partial \mathfrak{T}^k_j(t)}$ denotes the unit outer normal along $\partial \mathfrak{T}^k_j(t)$ which is tangential to
$\Gamma^k$ and $\mathfrak e$ denotes an edge of the boundary $\partial \mathfrak{T}^k_j(t)$.
We introduce the approximations
\begin{align}
 \int_{\mathfrak{T}^k_j(t^{k+1})}  u - \int_{\mathfrak{T}^k_j}  u \approx V^{k+1}_j u^{k+1}_j - V^{k}_j u^{k}_j \nonumber
\end{align}
and
\begin{align}
 \int_{t^k}^{t^{k+1}} \sum_{\mathfrak e \subset \partial \mathfrak{T}^k_j(t)} \int_{\mathfrak e} \f(\cdot,\cdot,u) \cdot \nu_{\partial \mathfrak{T}^k_j(t)}
          \approx  \tau  \sum_{ e \subset \partial {T}^k_j} g^k_{e, {T}^k_j}(u^k_j,u^k_{l(j, e)}),\nonumber
\end{align}
where $u^{k}_j \in\R$ represents the value of $u$ on $\mathfrak{T}^k_j$, $l(j,e)$ is the index of the simplex which shares the edge $ e $
with ${T}^k_j$ and the $g^k_{e,\partial {T}^k_j}$ are some numerical flux functions which depend on cell number, edge and time index.
We define the finite volume scheme by
\begin{algorithm}
\begin{align}
 u^{k+1}_j  &:= \frac{1}{V^{k+1}_j} \left( V^{k}_j u^{k}_j  \label{eqFVScheme}
       - \tau \sum_{ e \subset \partial {T}^k_j} g^k_{ e,{T}^k_j}(u^k_j,u^k_{l(j, e)}) \right),\\
u^0_j &:= \frac{1}{V^0_j}\int_{\mathfrak{T}^0_j}  u_0.  \label{eqFVSchemeIV}
\end{align}
\end{algorithm}
For implementation purposes we introduce the approximated total mass on a cell $m^k_j:=V^{k}_j u^{k}_j$.
In terms of these quantities the finite volume scheme (\ref{eqFVScheme}) and (\ref{eqFVSchemeIV}) reads
\begin{align}
 m^{k+1}_j  &:= m^{k}_j
       - \tau \sum_{ e \subset \partial {T}^k_j} g^k_{ e,{T}^k_j}\left(\frac{m^k_j}{V^{k}_j},\frac{m^k_{l(j, e)}}{V^{k}_{l(j, e)}}\right) ,\\
m^0_j &:= \int_{\mathfrak{T}^0_j}  u_0.
\end{align}

\begin{example}[Engquist-Osher Numerical Flux]\label{bspEOLF}
Using the notations from above we define for a simplex ${T}^k_j\in \mathcal T^k_h$ and an edge $e\subset \partial{T}^k_j$
\begin{align}
 c^{k}_{ e,{T}^k_j}( u)&:= \int_{e} f(\cdot,t^k, u) \cdot \nu_{\partial{T}^k_j},\\
 c^{k,+}_{ e,{T}^k_j}( u)&:=c^{k}_{ e,{T}^k_j}(0)+\int_0^{ u}\max \{{c^{k}_{ e,{T}^k_j}}^\prime(s),0\}ds,\\
 c^{k,-}_{ e,{T}^k_j}( u)&:=\int_0^{ u}\min \{{c^{k}_{ e,{T}^k_j}}^\prime(s),0\}ds.
\end{align}
Then an Engquist-Osher-flux is defined by
\begin{align}
 g^{k,\text{EO}}_{ e,{T}^k_j}(u,v):=c^{k,+}_{ e,{T}^k_j}(u)+c^{k,-}_{ e,{T}^k_j}(v).
\end{align}
\end{example}

\section{Numerical Experiments}

The finite volume scheme \eqref{eqFVScheme} and \eqref{eqFVSchemeIV} is validated by numerical experiments.
To this end we formulate a linear transport problem on a sphere whose radius decreases exponentially in time.
For this problem we know the exact solution and, thus, can compute some 
experimental orders of convergence (EOC).
As a second test problem we choose a nonlinear (Burgers-like) flux function $\f$ and state our problem
on an ellipsoid which develops a narrowness in time.

\vspace{0.2cm}
\textbf{Test Problem 1 (Linear)} We define
\begin{align*}
 r(t)&:= \exp(-t),\\
 \Gamma_t &:= \left\{ r(t) x \:|\: x \in \mathbb{S}^2 \right\},\\
\Phi(x,t) & := r(t)x,\\
 \mathbf V(r(t)x) &:= (-x^2,x^1,0)^{\text T} \quad \text{for } x=(x^1,x^2,x^3)\in \mathbb S^2,\\
f(u)&:= 2\pi u,\\
\f(x,t,u)&:= f(u) \mathbf V (r(t)x)
\end{align*}
and consider the initial value problem
\bea\label{eqConsLawTestProblem}
 \dot{u} + u\nabla_{\Gamma_t}\cdot v + \nabla_{\Gamma_t}\cdot \f(\cdot,\cdot,u) &=& 0  \quad\text{on}\quad G_T,\\
 u(x,0) &=& u_0(x)  \quad\text{on} \quad\Gamma_0. \label{eqIVTestProblem}
\eea
Since $v= \dot{r}(t) \nu $, where $\nu$ denotes the unit outer normal on $\mathbb S^2$, we have $\nabla_{\Gamma_t}\cdot v = \frac{2\dot{r}(t)}{r(t)}=-2$ (cf. \cite{dziuk1}).
As in \cite{m}, one sees that the last term on the left hand side in (\ref{eqConsLawTestProblem}) reads
in polar coordinates $(r,\varphi,\theta)$
\begin{align*}
 \frac{1}{r(t)} \frac{\partial}{\partial \varphi} f(u).
\end{align*}
Therefore (\ref{eqConsLawTestProblem}) is equivalent to
\begin{align}\label{eqConsLawTestProblemEquiv}
 \dot{u} - 2u   +   \exp(t) \frac{\partial}{\partial \varphi} f(u) = 0
              \;\text{ for } \varphi\in \mathbb T^1:=[0,2\pi] \text{ and } t>0,
 \end{align}
where $\varphi=0$ and $\varphi=2\pi$ are identified.
For given functions $\widetilde u_0: \mathbb T^1 \rightarrow \R$
and $\widehat u: (0,\pi) \rightarrow \R$ we define in polar coordinates
\begin{align*}
 \widetilde u (\varphi,t) &:=\exp(2t) \widetilde u_0(\varphi -  2\pi (\exp(t)-1)), \\
u(\varphi,\theta,t) &:= \widetilde u (\varphi,t) \widehat u(\theta),\\
u_0(\varphi,\theta)&:=\widetilde u_0 (\varphi) \widehat u(\theta).
\end{align*}
\begin{table}
\centering	
\caption{Experimental orders of convergence (EOC) for \textbf{Test Problem 1}, where $\bar h$ denotes the average diameter of the grid's elements.}
\label{tableEOC1}
\vspace{0.5cm}
\begin{tabular}{cccc}
  Elements & $\bar h$ & L1-Error & EOC \\
\hline
632&0.21605&1.86&	---\\
2,628&0.10613&1.53&	0.27\\
11,164&0.05145&1.16&	0.39\\
45,102&0.02557&0.76&	0.59\\
187,682&0.01251&0.49&	0.61\\
747,416&0.00627&0.30&	0.68\\
\end{tabular}
%\end{wrapfigure}
\end{table}

With these definitions one easily sees that $u$ is a solution of the initial value problem (\ref{eqConsLawTestProblem})-(\ref{eqIVTestProblem}).
For testing our code we define
\begin{align*}
  \widehat u (\theta) & := \sin^2(3\theta)\mathbbmss 1_{\left\{|\theta-\pi/2|<\pi/6\right\}}(\theta),\\
\widetilde u_0(\varphi)&:=\mathbbmss 1_{\left\{\varphi <\pi\right\} }(\varphi)
\end{align*}
and choose an Engquist-Osher numerical flux. 
For our computations we use surface grids that approximate the sphere $\mathbb S^2$. They consist of flat 
triangles whose nodes lay on $\mathbb S^2$.
We get the experimental orders of convergence which are
listed in Table \ref{tableEOC1}. 

\vspace{0.2cm}

\textbf{Test Problem 2 (Nonlinear)}
The results of three further experiments are illustrated in Figures \ref{pic:ellipsoid},
\ref{pic:newshock} and \ref{pic:newshockdivfree}, respectively. All three have the function
\begin{equation*}
u_0(x_1,x_2,x_3)=\cos^2(\pi\ (x_1+2))\mathbbmss 1_{\left\{x_1 < -3/2 \right\}}(x_1) 
\end{equation*}
as initial values.
For the first two experiments (see Figures \ref{pic:ellipsoid} and \ref{pic:newshock})
the flux function $f$ is constructed by taking a constant vector field which is pointing in direction
of the $x_1$-axis and projecting it on the hypersurface $\Gamma_t$. This flux function is not divergence-free.
Figure \ref{pic:ellipsoid} shows the numerical solution of a Burgers equation on an evolving
ellipsoid. You can see a shock that moves from left to right. In Figure \ref{pic:newshock} the same equation is 
considered, but due to fast change of geometry in the middle of the ellipsoid, the mass is compressed so fast that
a second shock riding on the first one is induced. Thus, this second shock is induced by the change of geometry.
For the third experiment (see Figure \ref{pic:newshockdivfree} ) the same parameters as in the second one are chosen,
only the flux function is different, which is chosen to be divergence-free.
Its construction is based on the following lemma which is a generalization of the one for the case of $\mathbb S^2$ developed by Ben-Artzi et al. \cite{BFL09} .

\begin{lemma}
Given a function $h=h(x,t,u)$ which is defined for $t\in[0,T]$ and $u\in \real$ in a neighbourhood of $\Gamma_t$,
then the flux function defined by $f(x,t,u) := \nu(x,t) \times \nabla h(x,t,u)$ is divergence-free, where  
the $x$-dependance of $f$ is assumed to $C^2$.
\end{lemma}
\begin{proof}
 For fixed $t\in[0,T],\ u\in\real$ we consider a portion $\gamma(t)$ of $\Gamma_t$ with smooth boundary $\partial \gamma(t)$. Then by the divergence theorem we have
 \begin{equation*}
  \aligned
   \int_{\gamma(t)} \nabla_{\Gamma_t} \cdot f(x,t,u) \dA 
   %&= \int_{\partial \gamma(t)} f(x,t,u) \cdot \mu(x,t) \dS  \\
   %&=\int_{\partial \gamma(t)} (\nu(x,t) \times \nabla h(x,t,u)) \cdot \mu(x,t) \dS  \\
   &=\int_{\partial \gamma(t)} (\mu(x,t) \times  \nu(x,t)) \cdot \nabla h(x,t,u) \dS.
   \endaligned
 \end{equation*}
 As $\mu(x,t) \times  \nu(x,t)$ is a unit tangent vector at $\partial \gamma(t)$ the integrand is the directional derivative along
 $\partial \gamma(t)$ und thus the integral vanishes for any smooth portion $\gamma(t)$.
\end{proof}

For the third experiment a flux corresponding to $h(x,t,u)=-20 x_3 u^2$ is chosen. The pictures from 
Figure \ref{pic:newshockdivfree} show the evolution of the numerical solution.
Here, as in Figure \ref{pic:newshock} a second shock is geometrically induced and overtakes the first one.

\begin {figure}
\begin {center}
\subfigure[t=0] {
\includegraphics[width=0.48 \linewidth]{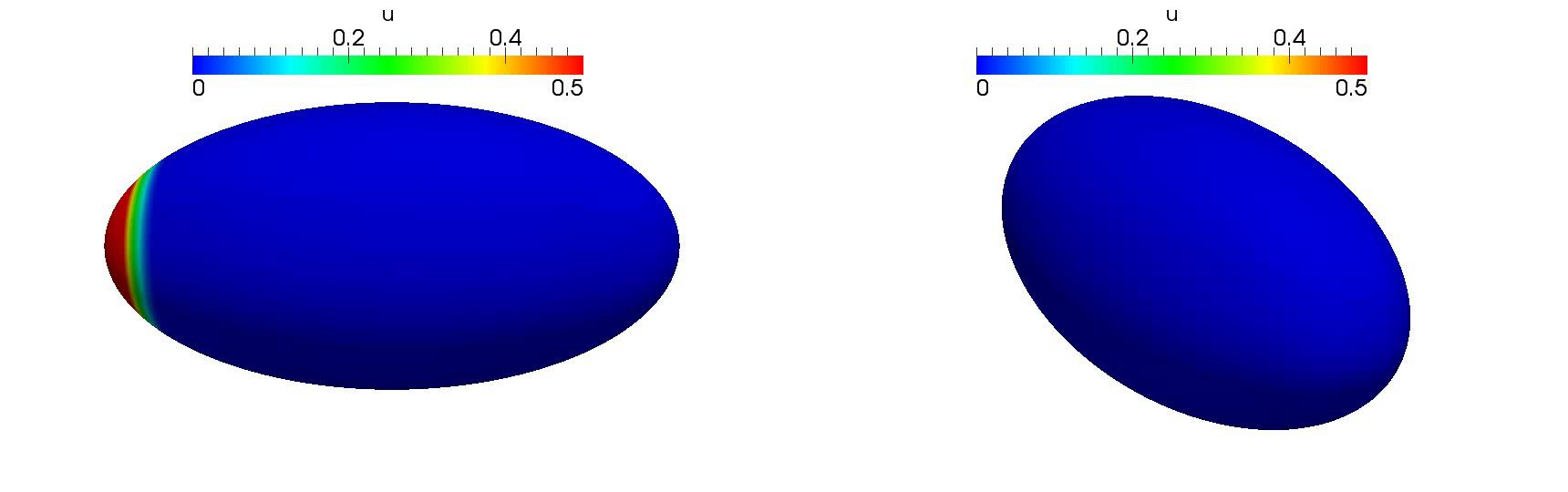}}
\subfigure[t=0.2T] {
\includegraphics[width=0.48 \linewidth]{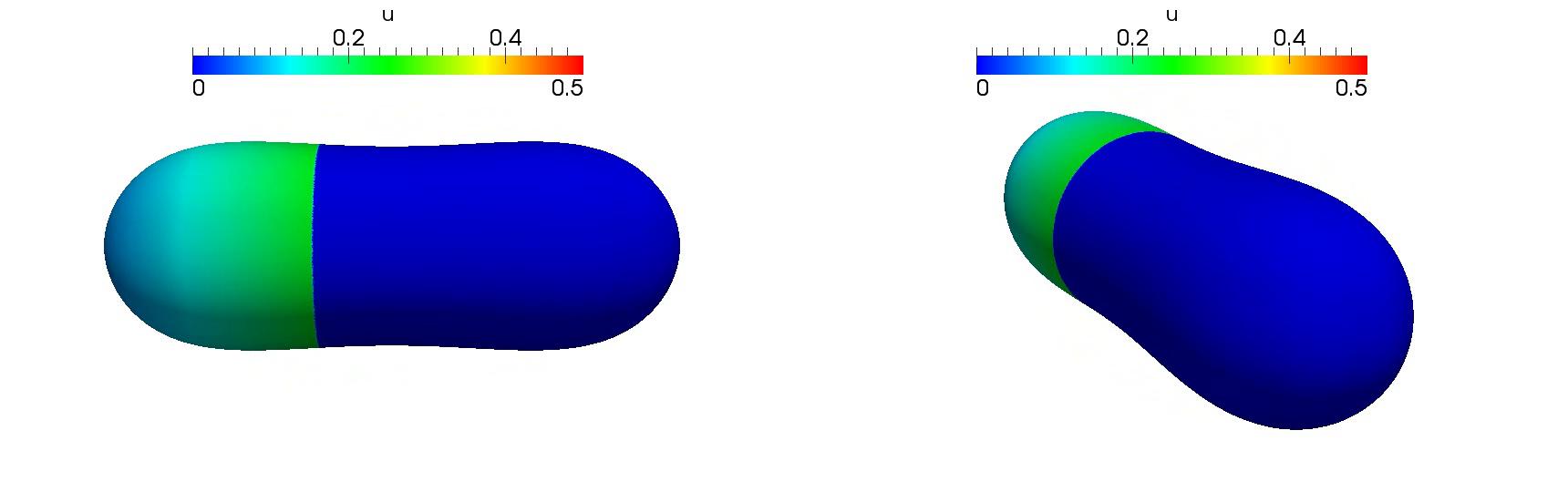}}\\
\subfigure[t=0.4T] {
\includegraphics[width=0.48 \linewidth]{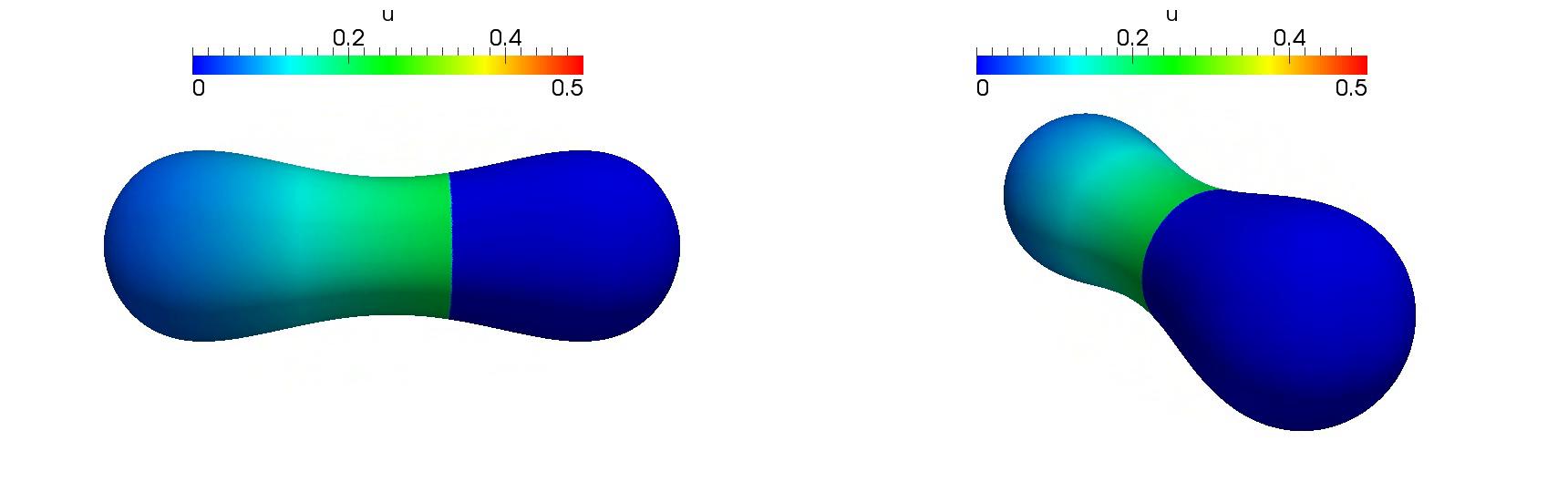}}
\subfigure[t=0.6T] {
\includegraphics[width=0.48 \linewidth]{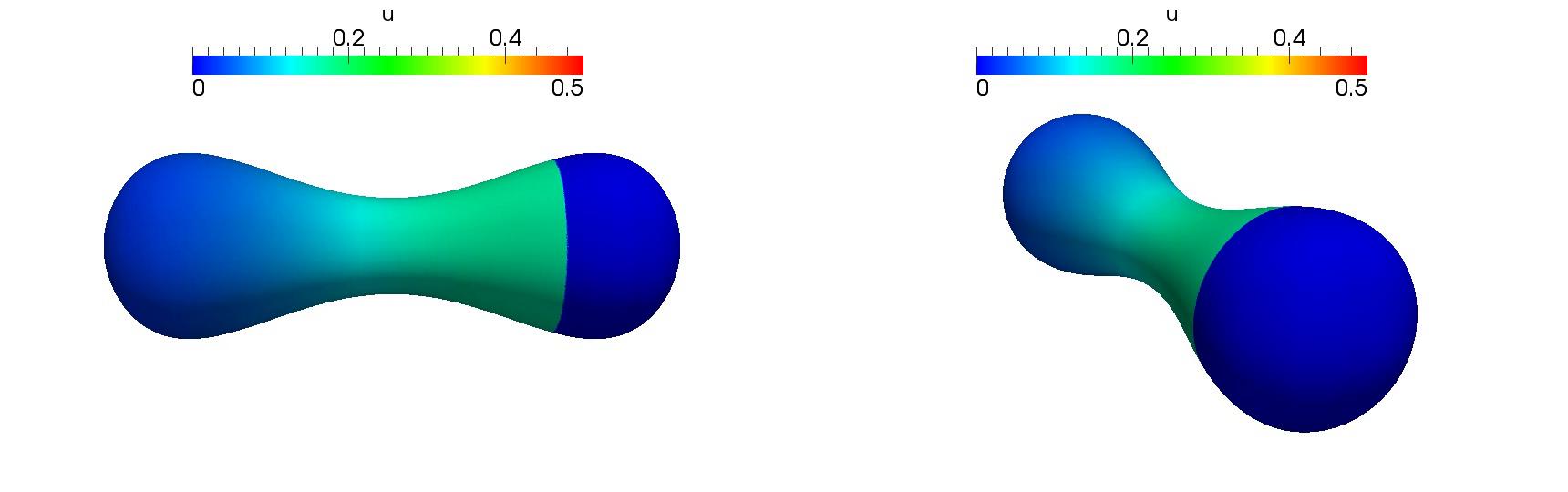}}\\
\subfigure[t=0.8T] {
\includegraphics[width=0.48 \linewidth]{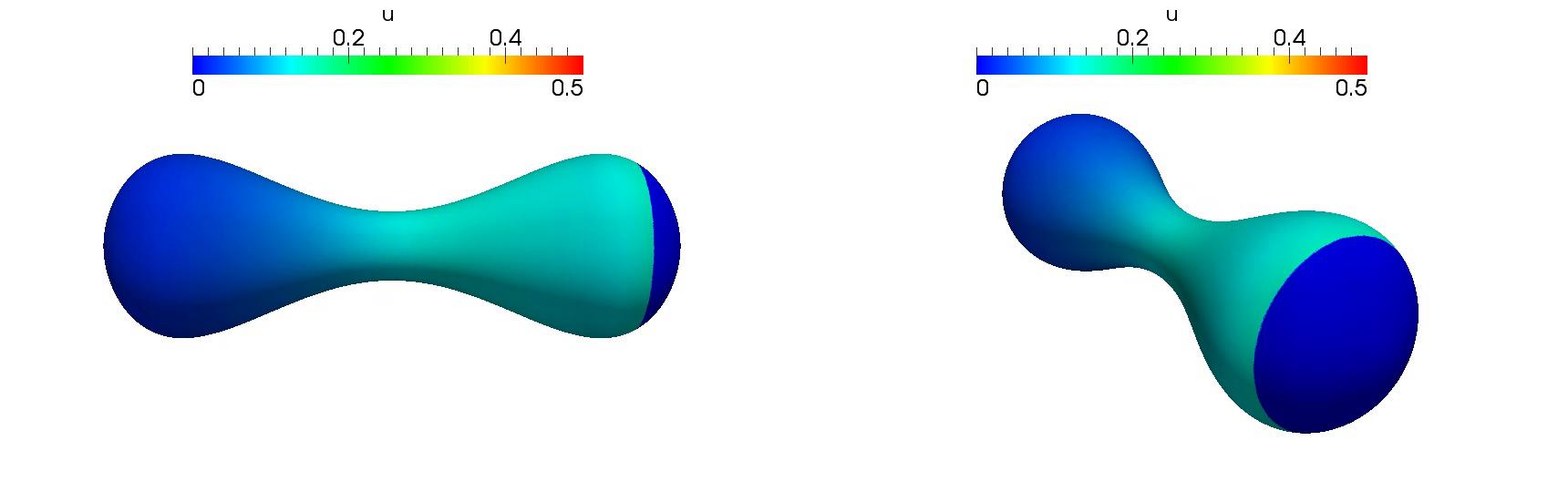}}
\subfigure[t=T] {
\includegraphics[width=0.48 \linewidth]{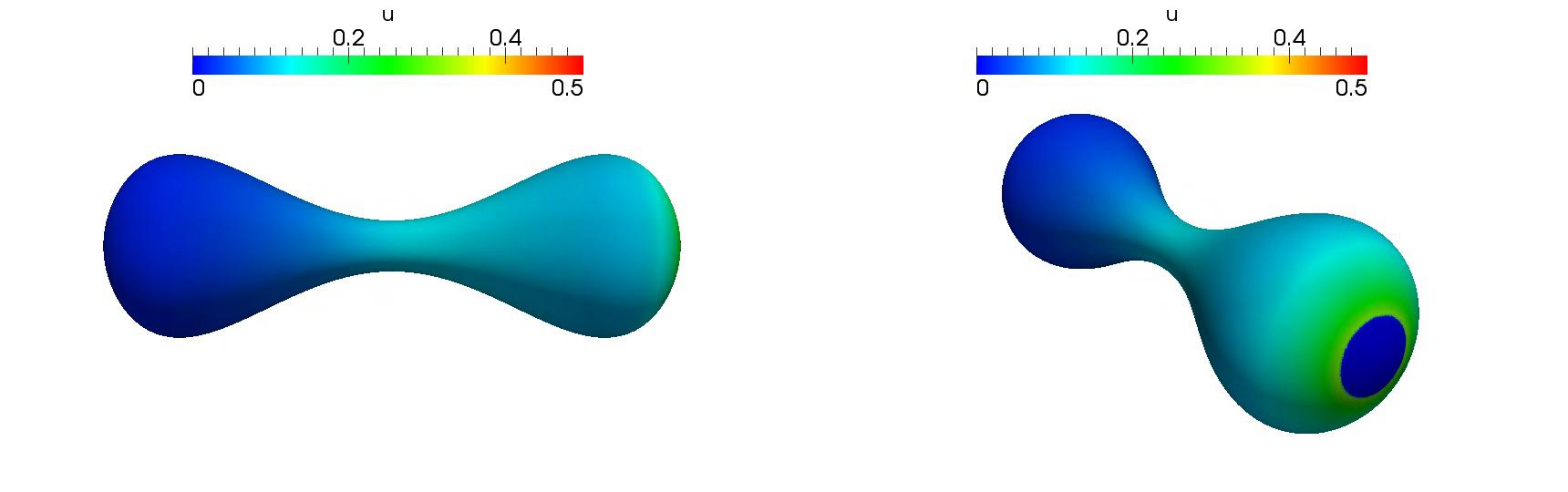}}\\
%\vspace {0.1cm}
\end {center}
\caption[Shock on Ellipsoid]{Burgers like shock on an evolving ellipsoid for several time steps. Here, $T$ denotes the end time.}
\label{pic:ellipsoid}
\end {figure}

\newpage
 \,\,\,

\begin {figure}[h!]
\begin {center}
\subfigure[t=0] {
\includegraphics[width=0.48 \linewidth]{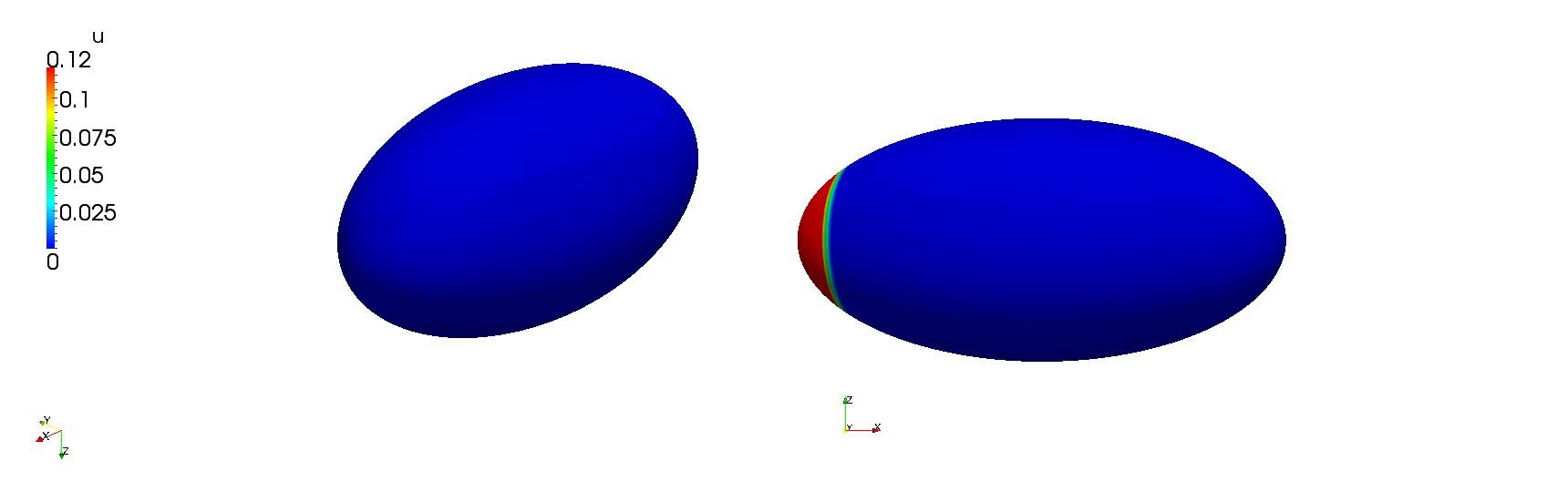}}
\subfigure[t=0.40T] {
\includegraphics[width=0.48 \linewidth]{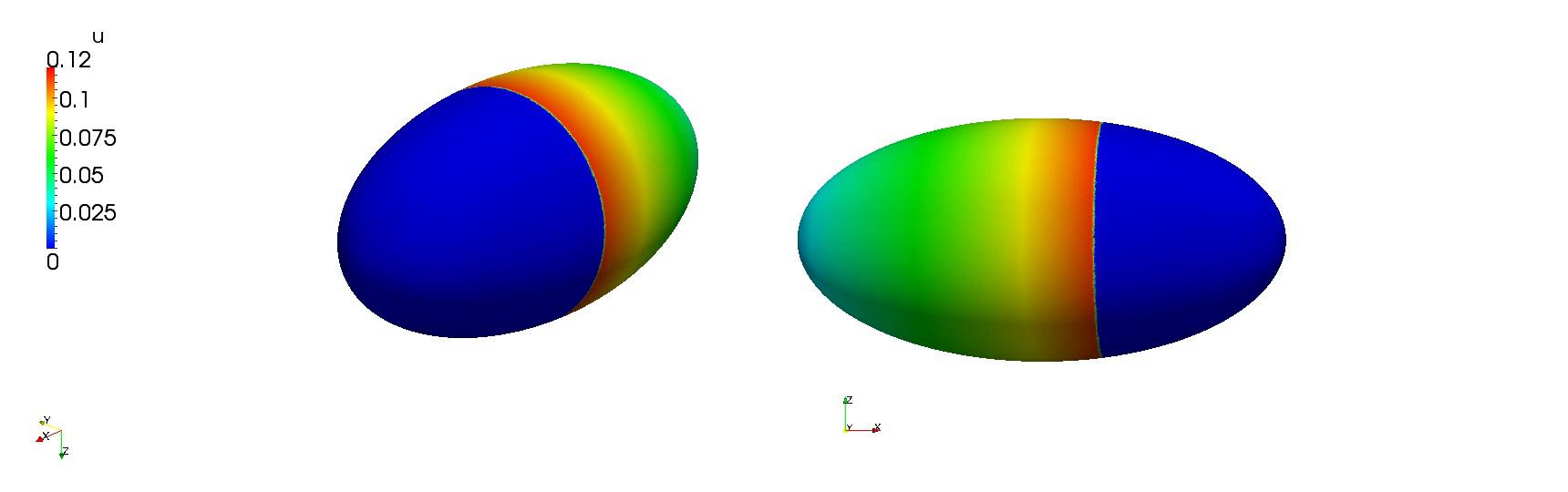}}\\
\subfigure[t=0.50T] {
\includegraphics[width=0.48 \linewidth]{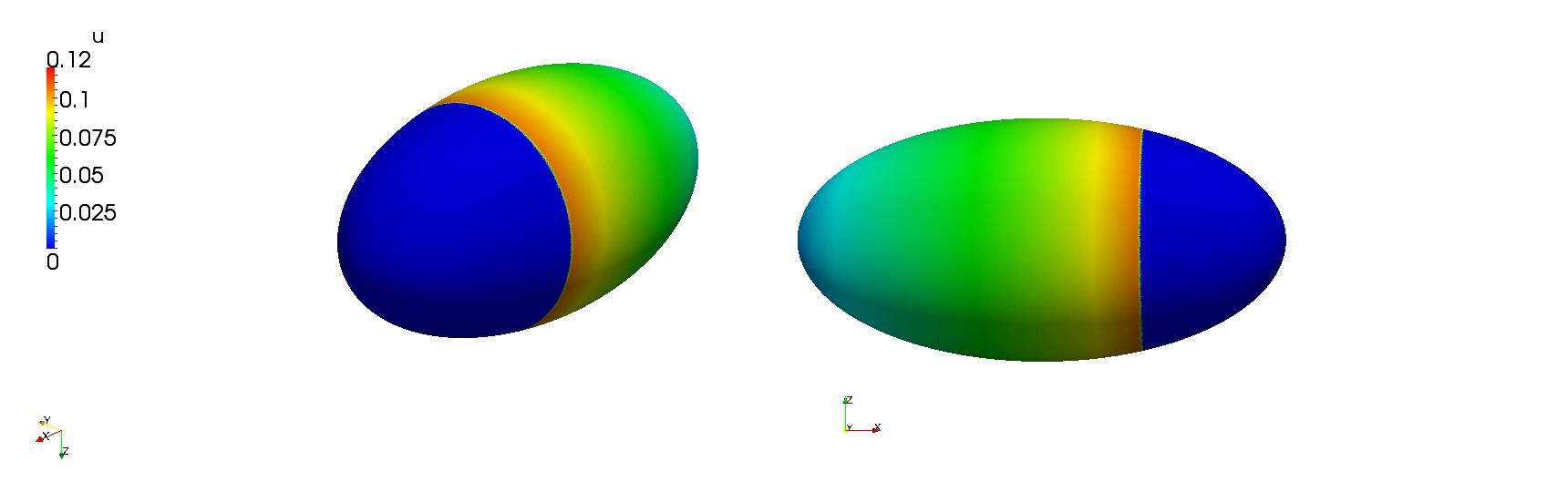}}
\subfigure[t=0.54T] {
\includegraphics[width=0.48 \linewidth]{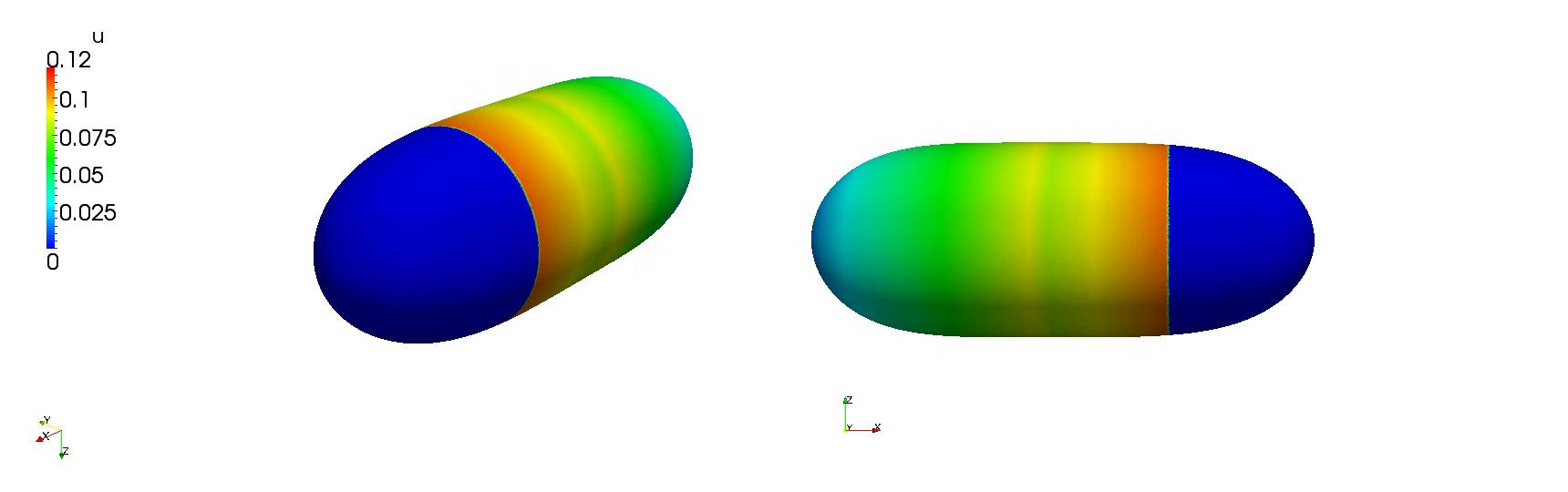}}\\
\subfigure[t=0.58T] {
\includegraphics[width=0.48 \linewidth]{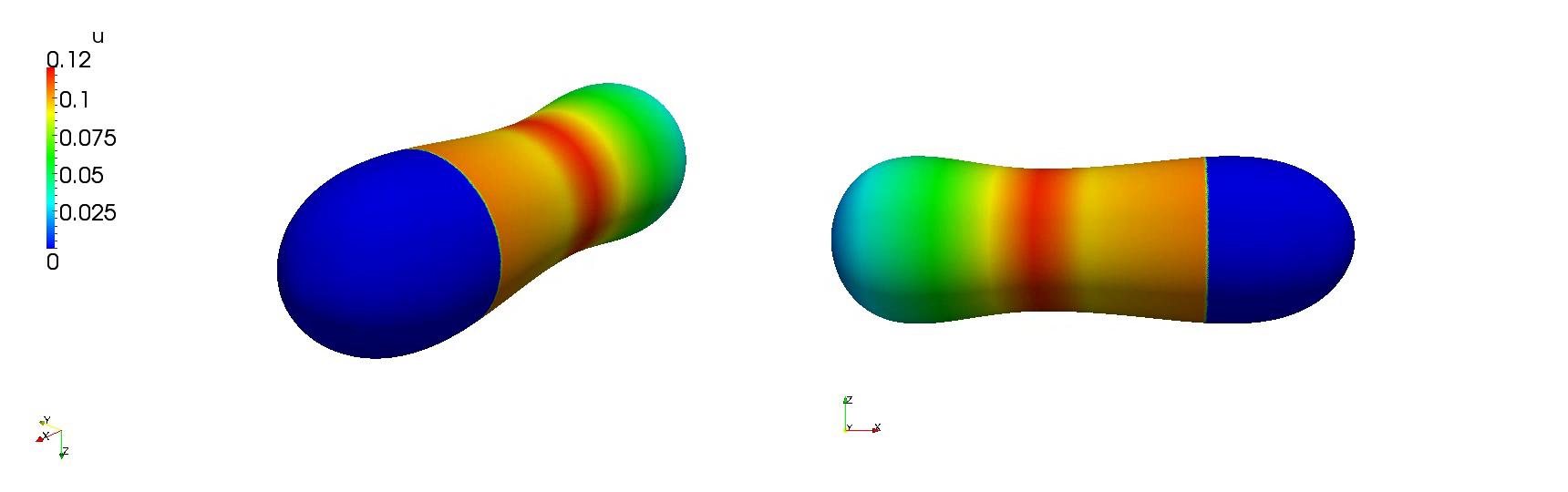}}
\subfigure[t=0.62T] {
\includegraphics[width=0.48 \linewidth]{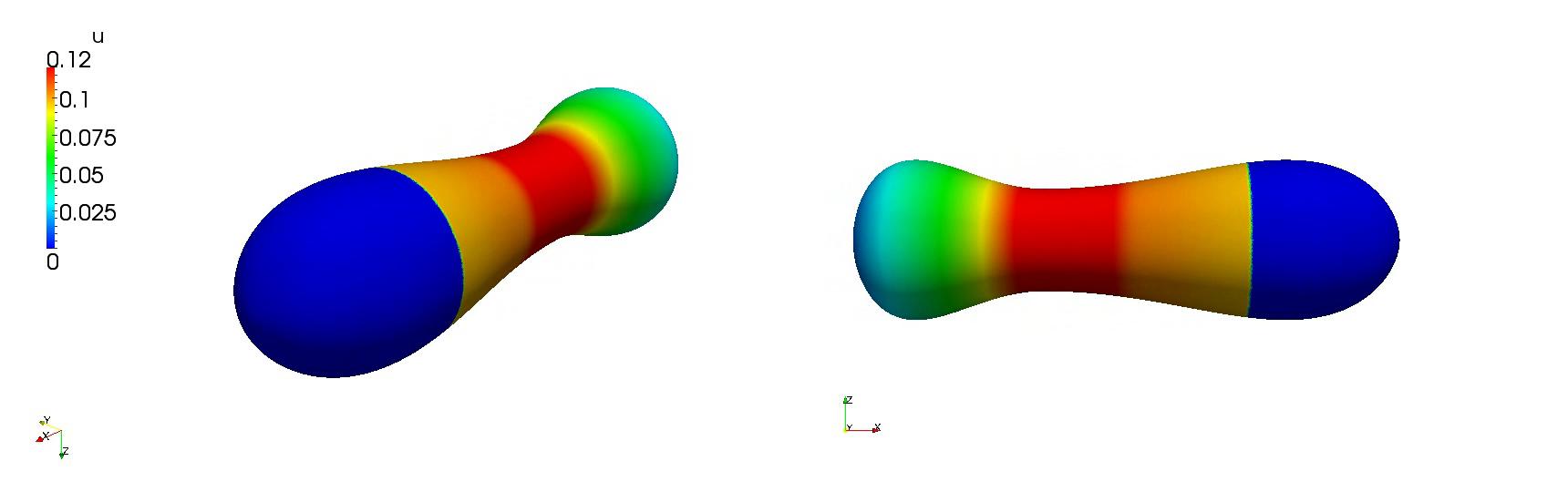}}\\
\subfigure[t=0.67T] {
\includegraphics[width=0.48 \linewidth]{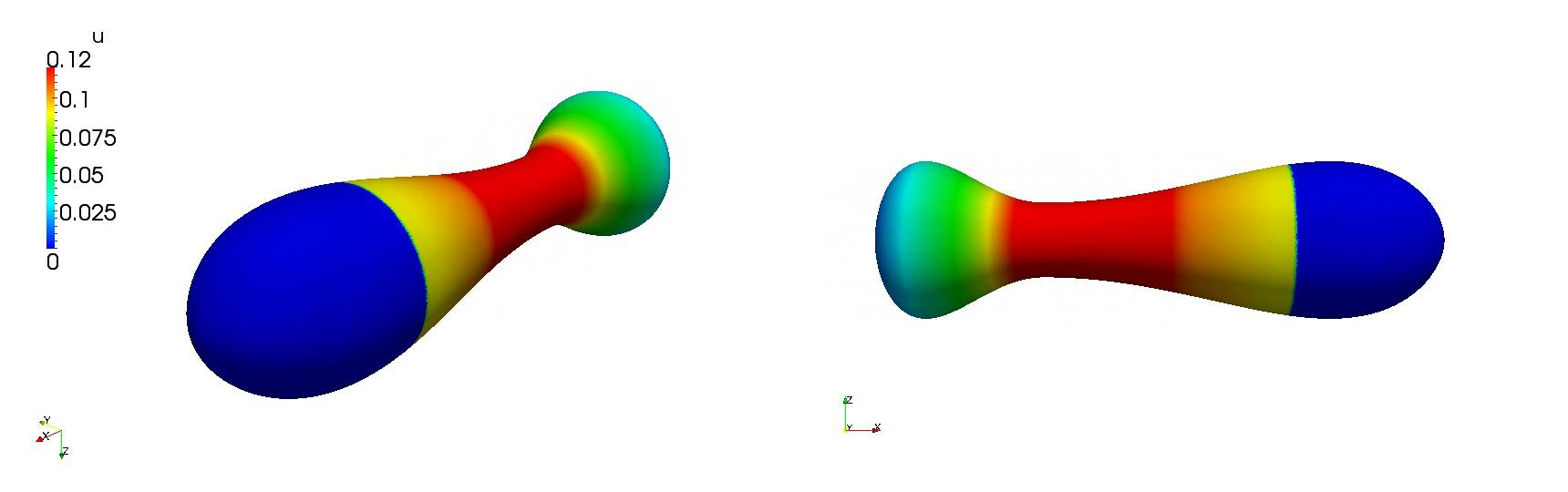}}
\subfigure[t=0.73T] {
\includegraphics[width=0.48 \linewidth]{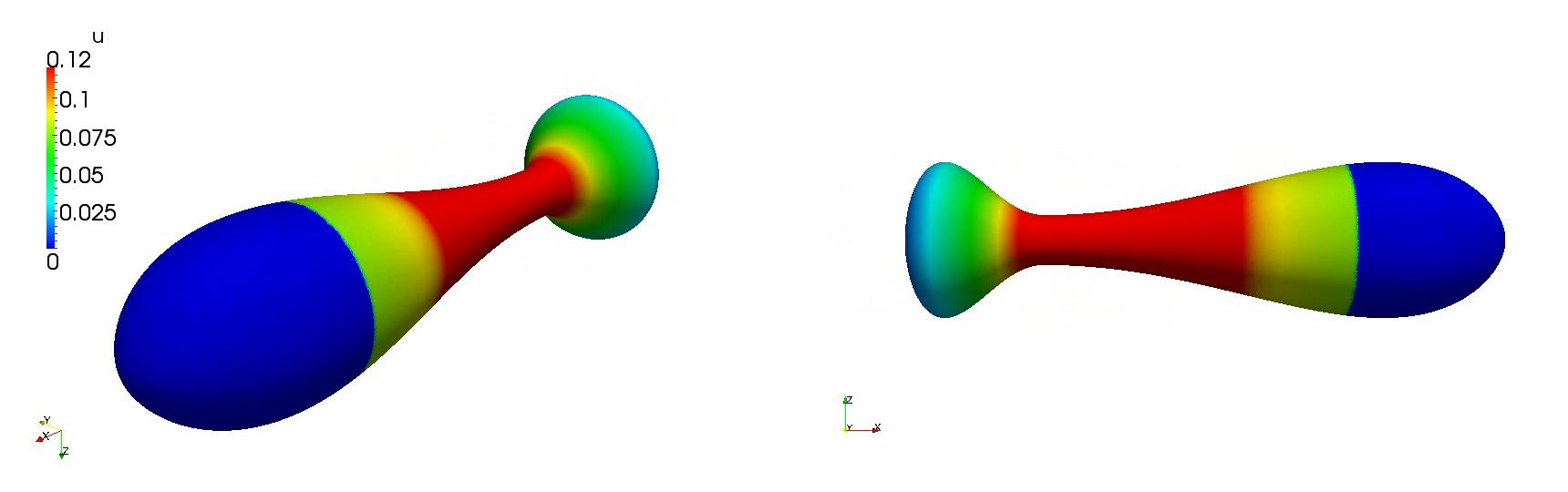}}\\
\subfigure[t=0.91T] {
\includegraphics[width=0.48 \linewidth]{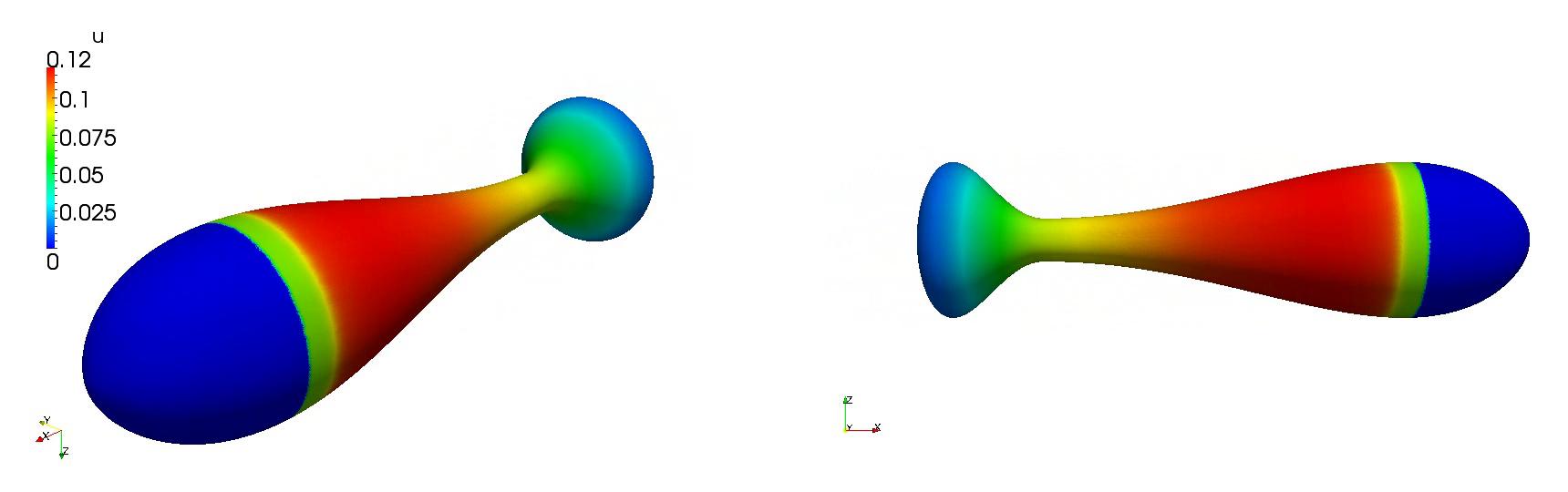}}
\subfigure[t=T] {
\includegraphics[width=0.48 \linewidth]{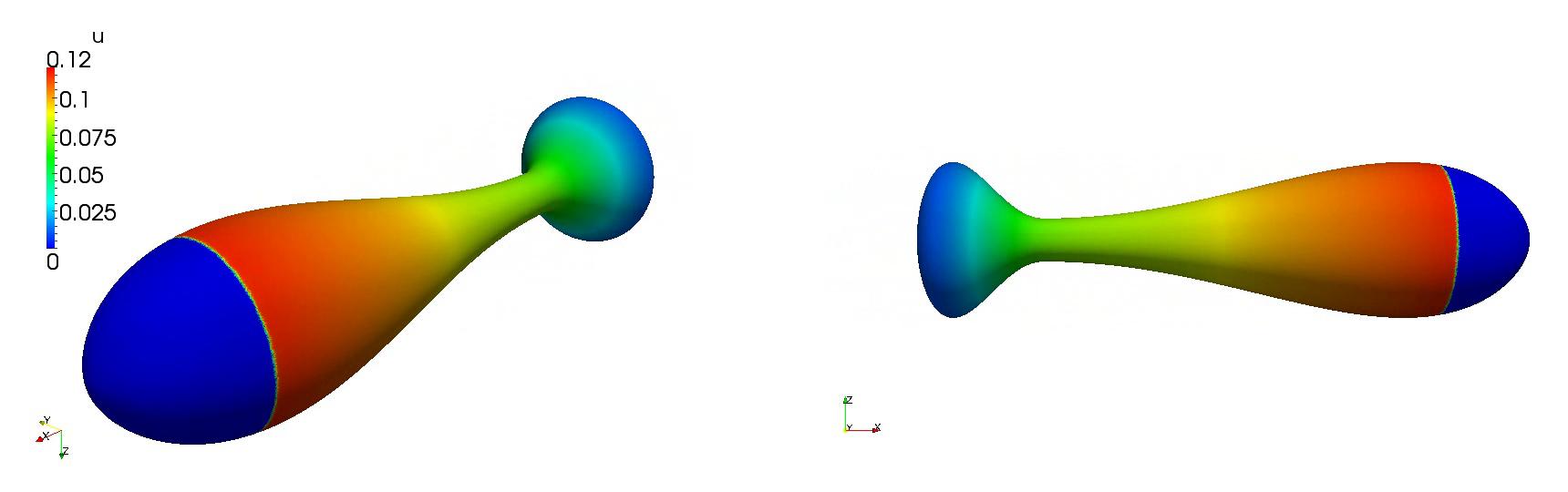}}
%\vspace {0.1cm}
\end {center}
\caption[Geometrically induced shock]{We see a Burgers like shock on an evolving ellipsoid.
Caused by the deformation of the ellipsoid a second shock is produced and overtakes the first one.
Here, $T$ denotes the end time.}
\label{pic:newshock}
\end {figure}

\,\,\,

\newpage
 \,\,\,

\begin {figure}[h!]
\begin {center}
\subfigure[t=0] {
\includegraphics[width=0.48 \linewidth]{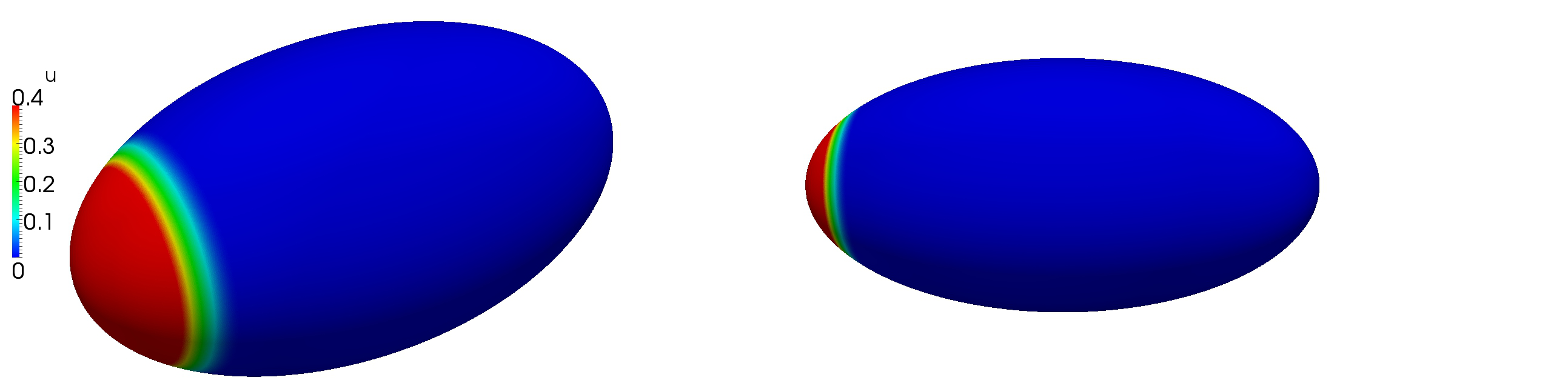}}
\subfigure[t=0.02T] {
\includegraphics[width=0.48 \linewidth]{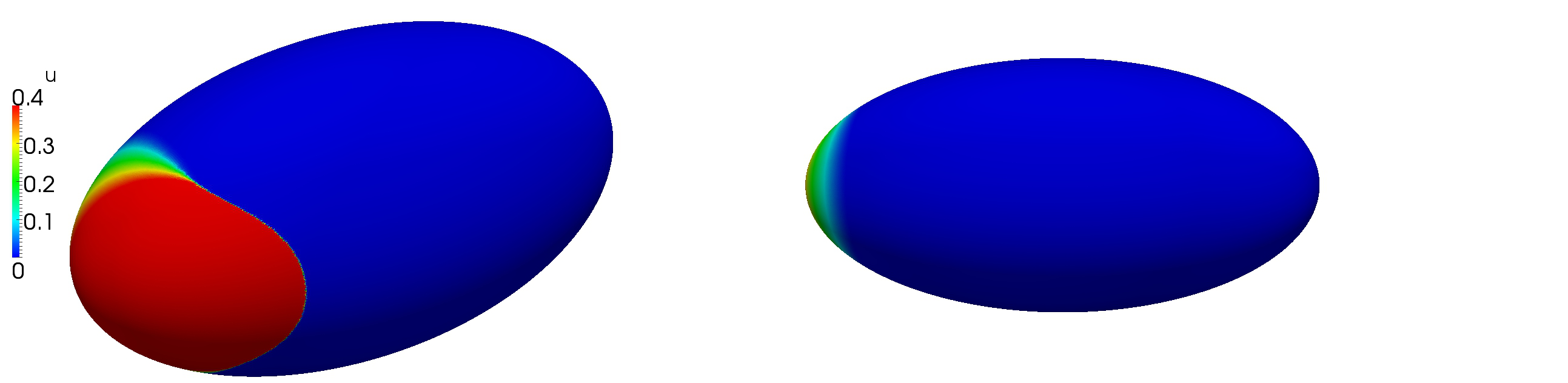}}\\
\subfigure[t=0.22T] {
\includegraphics[width=0.48 \linewidth]{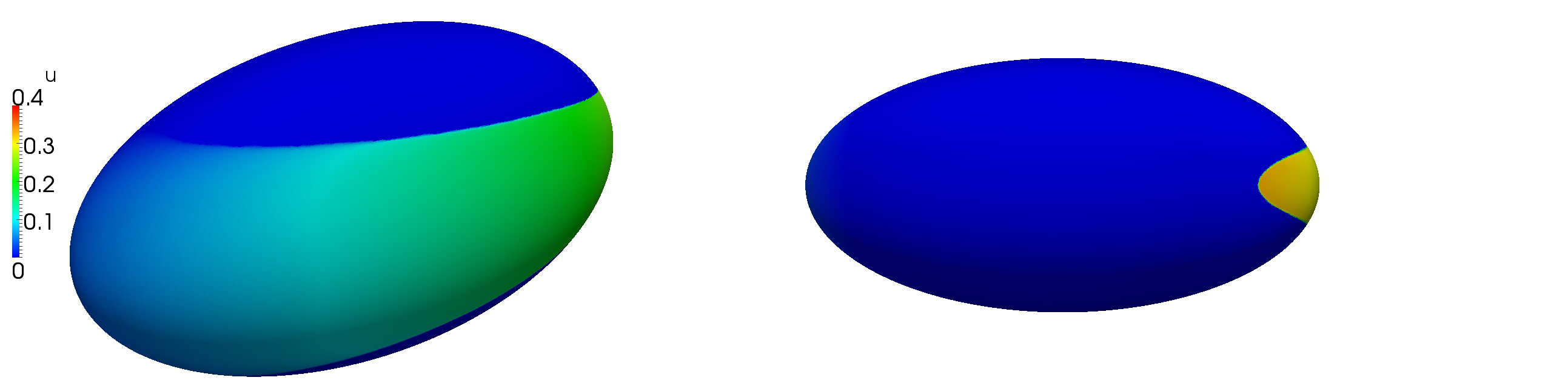}}
\subfigure[t=0.26T] {
\includegraphics[width=0.48 \linewidth]{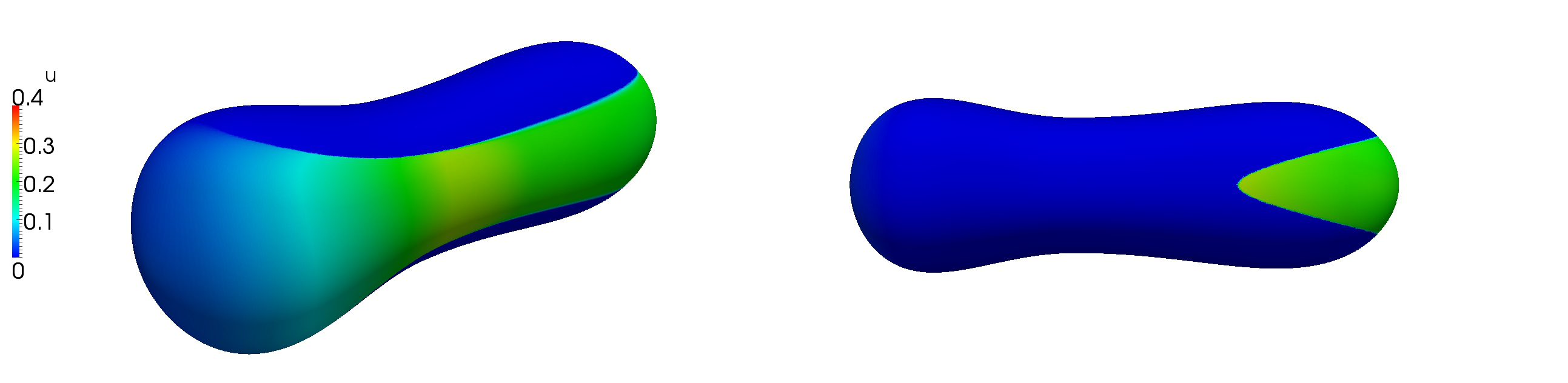}}\\
\subfigure[t=0.28T] {
\includegraphics[width=0.48 \linewidth]{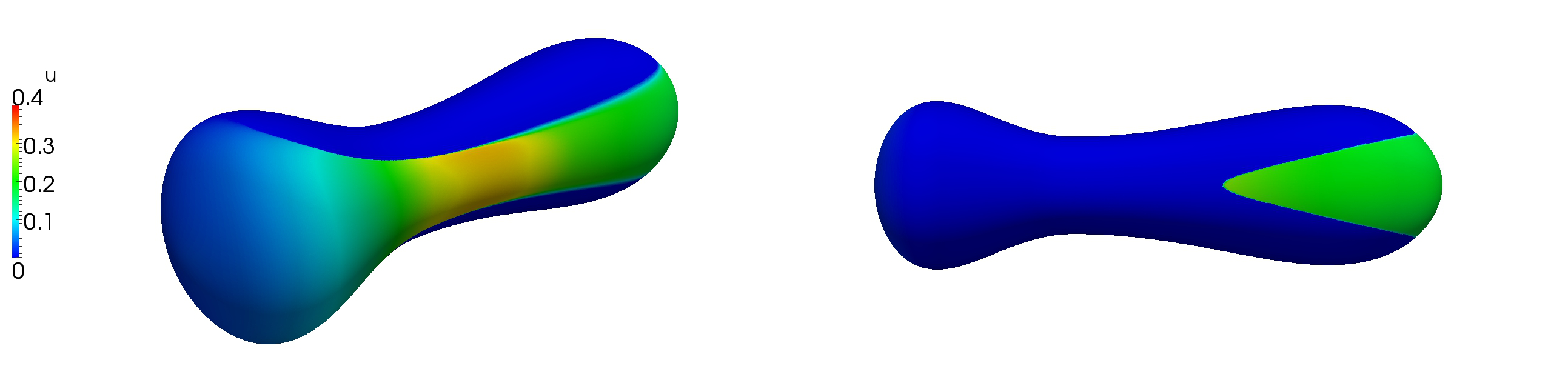}}
\subfigure[t=0.29T] {
\includegraphics[width=0.48 \linewidth]{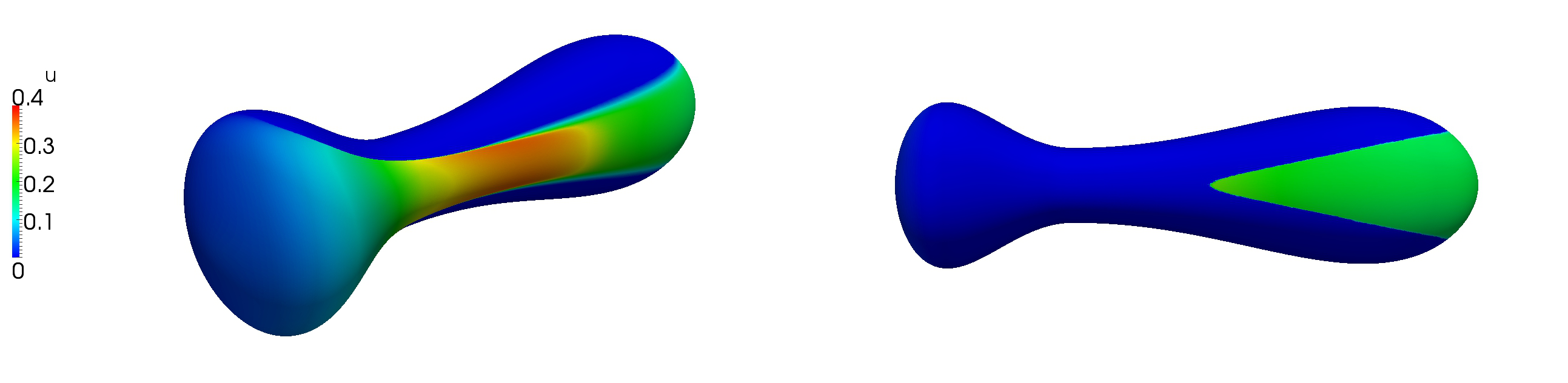}}\\
\subfigure[t=0.33T] {
\includegraphics[width=0.48 \linewidth]{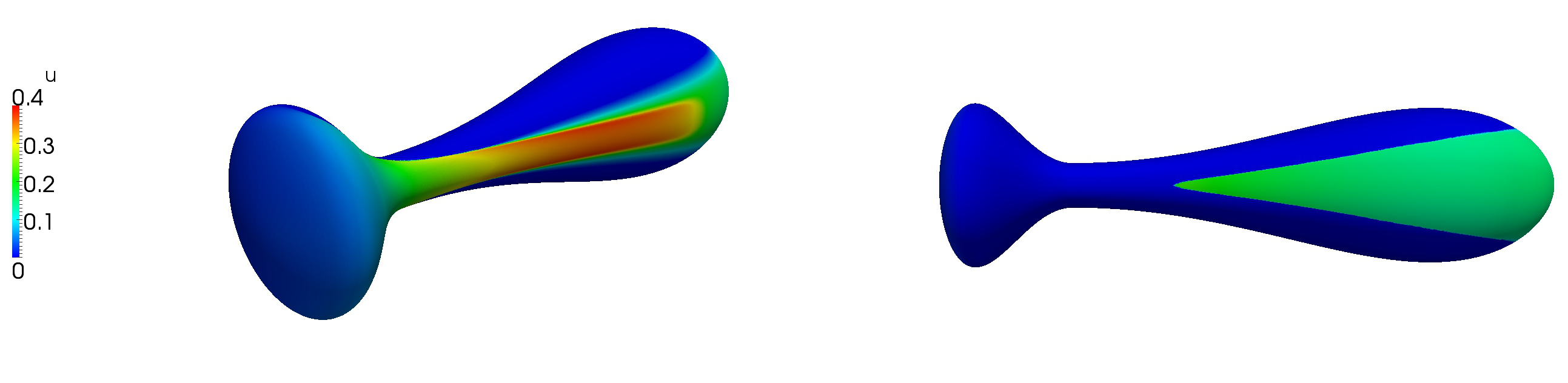}}
\subfigure[t=0.40T] {
\includegraphics[width=0.48 \linewidth]{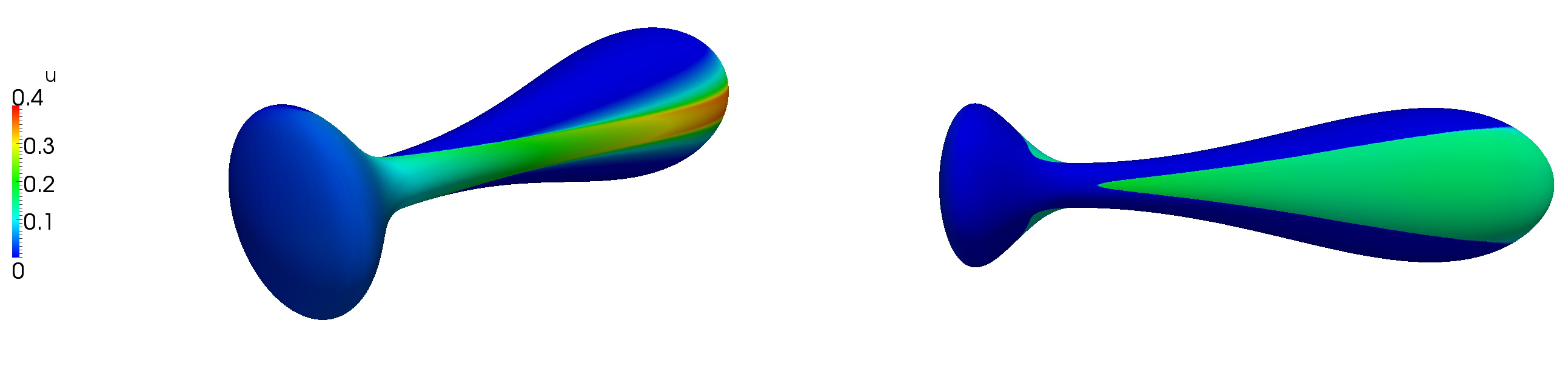}}\\
\subfigure[t=0.58T] {
\includegraphics[width=0.48 \linewidth]{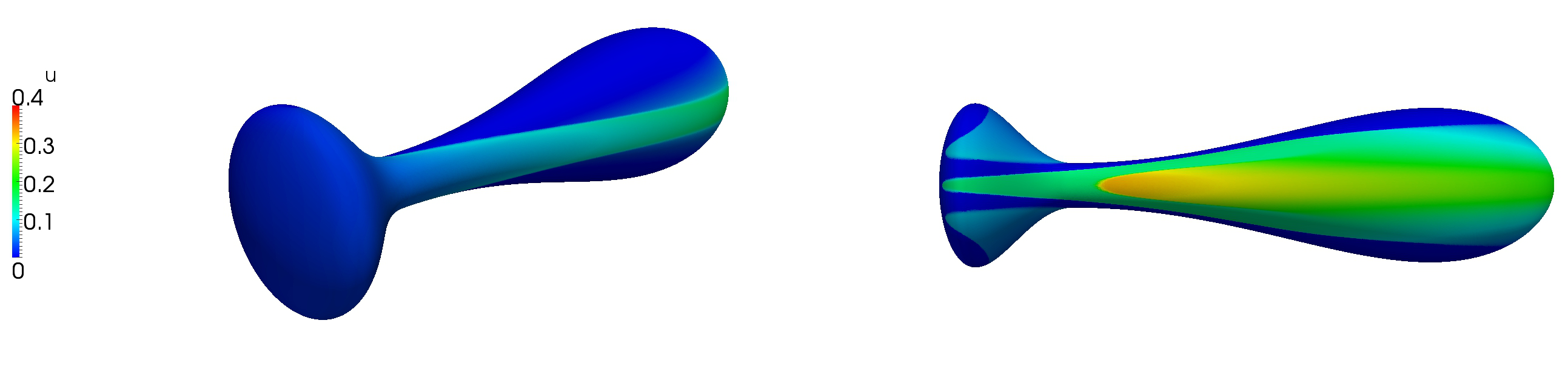}}
\subfigure[t=0.64T] {
\includegraphics[width=0.48 \linewidth]{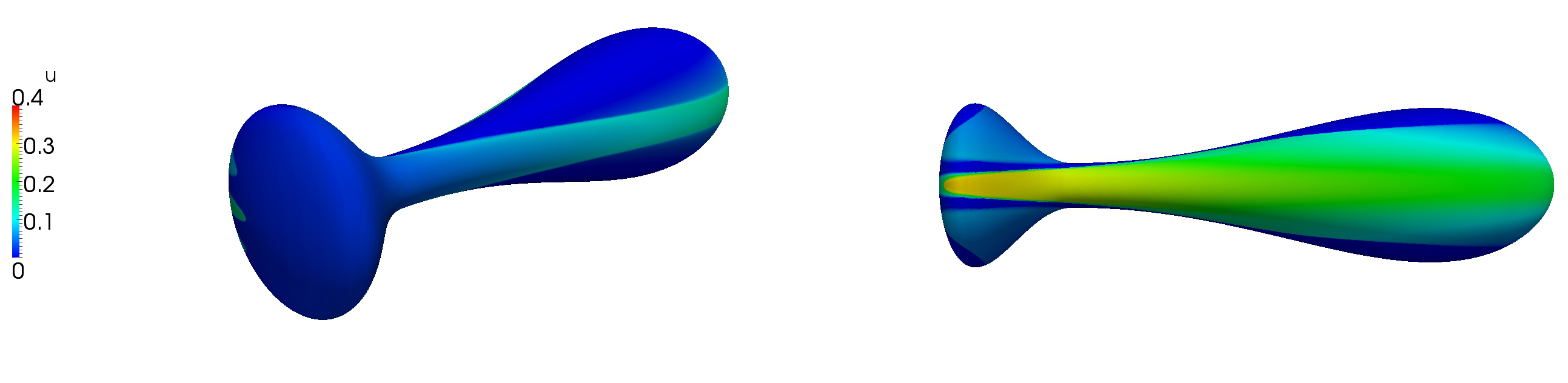}}\\
\subfigure[t=0.79T] {
\includegraphics[width=0.48 \linewidth]{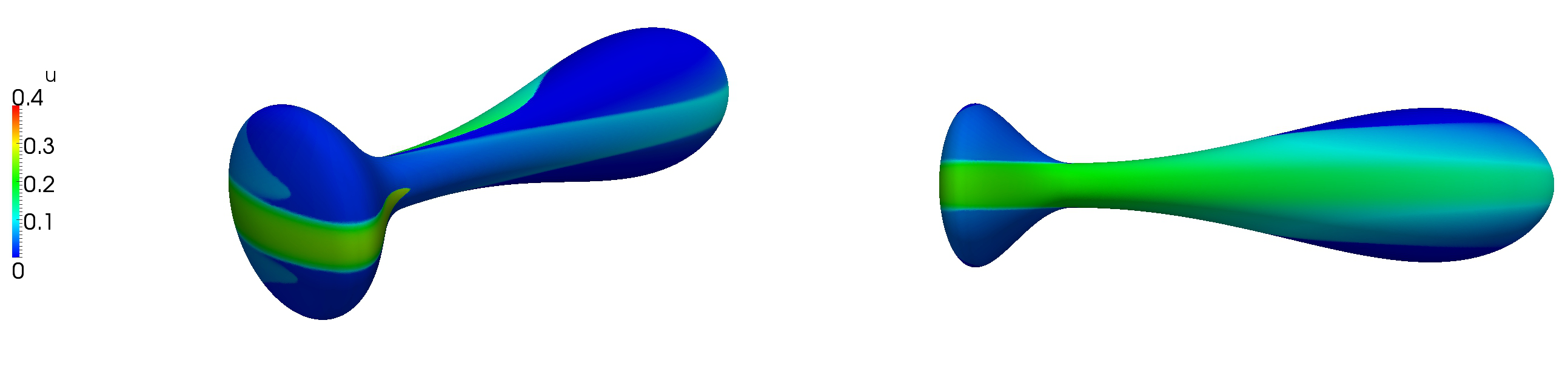}}
\subfigure[t=T] {
\includegraphics[width=0.48 \linewidth]{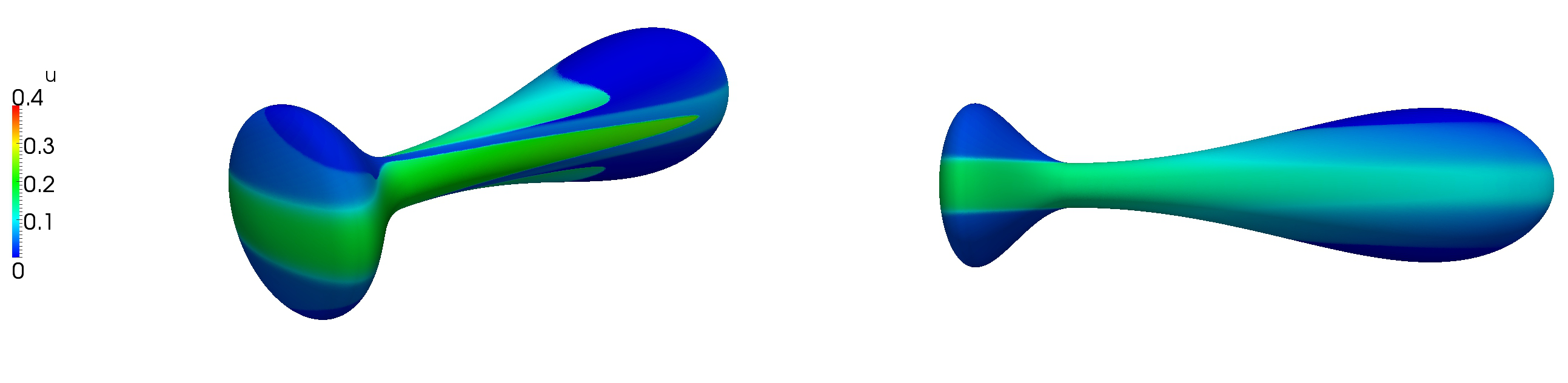}}\\
%\vspace {0.1cm}
\end {center}
\caption[Geometrically induced shock with divergence-free flux function]{As in Figure \ref{pic:newshock} a second shock is
geometrically induced and overtakes the first one. Here, the flux function $f$ is divrgence-free. $T$ denotes the end time.}
\label{pic:newshockdivfree}
\end {figure}

\,\,\,

\section{Appendix}
Proof of Lemma \ref{dotgradLemma}
\begin{proof}
It is sufficient to prove this relation locally. For this we
use a parametrization $X=X(\theta,t)$, $\theta\in \mathcal{O}$
over some open set $\mathcal{O}\subset\R^n$.
We write $U(\theta,t)=u(X(\theta,t),t)$ and $V(\theta,t)=v(X(\theta,t),t)$.
We use the notation $g_{ij}=X_{\theta_i}\cdot X_{\theta_j}$. The uppering
of the indices have to be understood in the usual sense as inverting the
matrix: $(g^{ij})_{i,j=1,\dots,n}=(g_{ij})_{i,j=1,\dots,n}^{-1}$.

Then, by definition of the material derivative (\ref{MatDer}) we have
\begin{eqnarray*}
\lefteqn{
\left(\ud_l u\right){\dot{~}}\circ X=\frac{\partial}{\partial t}\left(\ud_lu\circ X\right)
=\frac{\partial}{\partial t}\left(g^{ij}U_{\theta_j}X^l_{\theta_i}\right)  }\\
&&=\frac{\partial g^{ij}}{\partial t} U_{\theta_j}X^l_{\theta_i}+
g^{ij}U_{t\theta_j}X^l_{\theta_i}+g^{ij}U_{\theta_j}X^l_{t\theta_i}
\end{eqnarray*}
We have that $U_t=\dot u\circ X$ and $X_t=V=v\circ X$. With the relation
\begin{equation*}
g^{ij}_t=-g^{ik}g^{jm}\left(V_{\theta_k}\cdot X_{\theta_m}+X_{\theta_k}\cdot V_{\theta_m}\right)
\end{equation*}
we get
\begin{eqnarray*}
\lefteqn{ \left(\ud_l u\right)^{\dot{~}}\circ X=
-g^{ik}g^{jm}\left(V^r_{\theta_k}X^r_{\theta_m}+X^s_{\theta_k}V^s_{\theta_m}\right)
U_{\theta_j}X^l_{\theta_i} + g^{ij}U_{t\theta_j}X^l_{\theta_i}
+g^{ij}U_{\theta_j}X^l_{t\theta_i}}\\
&&=-g^{ik}V^r_{\theta_k}X^l_{\theta_i}g^{jm}X^r_{\theta_m}U_{\theta_j}
-g^{ik}X^s_{\theta_k}X^l_{\theta_i}g^{jm}V^s_{\theta_m}U_{\theta_j}
+g^{ij}U_{t\theta_j}X^l_{\theta_i}+g^{ij}U_{\theta_j}V^l_{\theta_i}\\
&&=-\ud_l u v_r\circ X \ud_r u\circ X - \ud_s x_l\circ X \nabla_\Gamma v_s\cdot \nabla_\Gamma u \circ X
+\ud_l \dot u \circ X + \nabla_\Gamma u\cdot \nabla_\Gamma v_l\circ X.
\end{eqnarray*}
Because of $\ud_s x_l=P_{sl}=\delta_{sl}-\nu_s\nu_l$ we finally arrive at
\begin{eqnarray*}
\lefteqn{
\left(\ud_l u\right){\dot{~}}
=\ud_l \dot u -\ud_l v_r \ud_r u - P_{sl}\ud_r v_s \ud_r u + \ud_r u \ud_r v_l }\\
&&= \ud_l \dot u - \ud_r u\left( \ud_l v_r + P_{sl}\ud_r v_s -\ud_r v_l\right).
\end{eqnarray*}
The Lemma is proved.
\end{proof}

\end{document}